\documentclass[12pt]{amsart}
\usepackage{amsmath,amssymb,mathrsfs,hyperref,color}

\usepackage[nameinlink]{cleveref}
\hypersetup{colorlinks={true},linkcolor={black},citecolor=black}
\crefformat{equation}{(#2#1#3)}

\usepackage{mathtools}
\usepackage[x11names]{xcolor}
\newtagform{blue}{\color{black}(}{)}

\usepackage[normalem]{ulem}

\definecolor{fgreen}{rgb}{0.13, 0.55, 0.13}

\usepackage{paralist}

\usepackage{subfigure}
\usepackage{float}
\usepackage{times}
\usepackage{tikz}
\usepackage{cases}
\usetikzlibrary{intersections}
\usepackage{xcolor}
\usepackage[latin1]{inputenc}
\usepackage{bbm}
\usetikzlibrary{shadings}
\usetikzlibrary[intersections]
\usetikzlibrary{arrows,shapes}

\makeatletter
\def\setliststart#1{\setcounter{\@listctr}{#1}%
  \addtocounter{\@listctr}{-1}}
\makeatother

\makeatletter
\@addtoreset{figure}{section}
\makeatother

\usepackage{calc}
\newtheorem{theorem}{Theorem}[section]
\newtheorem{lemma}[theorem]{Lemma}
\newtheorem{proposition}[theorem]{Proposition}
\newtheorem{corollary}[theorem]{Corollary}
\newtheorem{remarks}[theorem]{Remark}
\newtheorem{example}[theorem]{Example}

\numberwithin{equation}{section}

\newcommand{\R}{\mathbb{R}}

\newcommand{\N}{\mathbb{N}}

\newcommand{\PP}{\mathscr{P}}

\newcommand{\C}{\mathcal{C}}
\newcommand{\GG}{\mathsf{G}}
\newcommand{\TT}{\mathcal{T}}
\newcommand{\g}{\mathsf{g}}

\newcommand{\X}{\mathcal{X}}

\newcommand{\diver}{\textrm{div}}

\DeclareMathOperator*{\eps}{\varepsilon}
\DeclareMathOperator*{\SR}{SR}

\makeatletter
\def\moverlay{\mathpalette\mov@rlay}
\def\mov@rlay#1#2{\leavevmode\vtop{%
   \baselineskip\z@skip \lineskiplimit-\maxdimen
   \ialign{\hfil$\m@th#1##$\hfil\cr#2\crcr}}}
\newcommand{\charfusion}[3][\mathord]{
    #1{\ifx#1\mathop\vphantom{#2}\fi
        \mathpalette\mov@rlay{#2\cr#3}
      }
    \ifx#1\mathop\expandafter\displaylimits\fi}
\makeatother

\title[Semi-linear equations and MFG]{Semi-linear parabolic equations on homogenous Lie groups arising from mean field games}
\author{Paola Mannucci \and Claudio Marchi \and Cristian Mendico}
\address{Dipartimento di Matematica "Tullio Levi-Civita", Universit\`a degli studi di Padova,  Via Trieste 63 - 35121 Padova}
\email{mannucci@math.unipd.it; claudio.marchi@unipd.it}
\address{Universit\`a degli studi di Roma Tor Vergata, Via della Ricerca Scientifica 1 - 00133 Roma}
\email{mendico@axp.mat.uniroma2.it}

\date{\today}
\thanks{The authors were partially supported by Istituto Nazionale di Alta Matematica, INdAM-GNAMPA project 2022, CUP E55F22000270001, and INdAM-GNAMPA project 2023, CUP E53C22001930001. The third author was also partially supported by the MIUR Excellence Department Project awarded to the Department of Mathematics, University of Rome Tor Vergata, CUP E83C23000330006. \\
The authors wish to thank Davide Barilari for fruitful discussions and comments. 
}

\begin{document}
\usetagform{blue}
\maketitle

\begin{abstract}
The existence and the uniqueness of solutions to some semilinear parabolic equations on homogeneous Lie groups, namely, the Fokker-Planck equation and the Hamilton-Jacobi equation, are addressed. The anisotropic geometry of the state space plays a crucial role in our analysis and creates several issues that need to be overcome. Indeed, the ellipticity directions span, at any point, subspaces of dimension strictly less than the dimension of the state space. Finally, the above results are used to obtain the  short-time existence of classical solutions to the mean field games system defined on an homogenous Lie group. 
\end{abstract}

\tableofcontents

\section{Introduction}

In the last decades, the existence and the regularity of solutions to linear, semi-linear and fully nonlinear equations defined on sub-Riemannian structures have attracted a lot of attention. Besides their own interests, these equations arise from several models such as diffusion process, control theory and human vision (we refer, for instance, to \cite{2014}, \cite{Baspinar_2020}). In all these settings such equations have in common the fact that their ellipticity directions span, at any point, subspaces of dimension strictly less than the dimension of the state space and all the remaining directions are recovered from commutators. It follows that the underlying geometric structure of the state space is of anisotropic type and this plays a crucial role in the analysis of solutions to elliptic and parabolic equations.

In the elliptic case, several results are present in literature on divergence form operators both for linear and semi-linear equations: Harnack's inequality, regularity results, existence and size estimates of the Green's function can be found, for instance, in \cite{Franchi_1995}, \cite{Lu_1992}. To the authors' knowledge, only few results are available for non-divergence form operators, for which we refer to  \cite{Bonfiglioli_2003},\cite{Bramanti_2000}, \cite{Polidoro_1994}. The parabolic counterpart, despite its great relevance from the applicative viewpoint, has been less investigated, and we mention \cite{Alexopoulos_2002} and references therein. A general and self-contained introduction to the subject can be found in  \cite{Bramanti_2010}, \cite{Bonfiglioli}, \cite{Alexopoulos_2002}.

In this work we focus our attention on two semi-linear parabolic equations posed on an homogenous Lie group: Hamilton-Jacobi equation and Fokker-Plank equation, both arising from mean field games (MFG) theory. In particular, we are interested in showing the existence of classical solutions to such equations in order to obtain, as a consequence, the existence of classical solutions to the MFG system. 
In the framework of evolutive noncoercive MFG, we quote here the results in \cite{Mannucci_2020, Mannucci_2022, Achdou_2022-1, Achdou_2022, bib:CM, Mendico} but in these cases a key assumption is that the coefficients in the dynamics grow at most linearly.
Up to the authors' knowledge, no results are available for parabolic MFG problem in unbounded domain for general Lie groups. 

Next, we give an overview of the main results proved in this paper and the strategy of proofs.

\subsection{Main results and strategy of proof}

\subsubsection{Fokker-Plank equation}

The first problem we address is the existence of classical solutions to the Fokker-Planck equation 
\begin{equation}\label{eq:FP1}
\begin{cases}
\partial_{t} \rho - \sigma \Delta_{\GG} \rho - \text{div}_{\GG}(b(t, x)\rho)=0, & (t, x) \in (0,T) \times \R^{d}
\\
\rho(0,x)=\rho_{0}(x),& x\in\R^d
\end{cases}
\end{equation}
where $T>0$, $\sigma>0$, $\Delta_{\GG}$ and $\text{div}_{\GG}$ are respectively the horizontal Laplacian and the horizontal divergence with respect of a family of left invariant vector fields $\{X_i\}_{i=1,\dots m}$ associated to a homogeneous Lie group (see Section  \ref{subell}) and the drift $b$ is regular and bounded w.r.t. these vector fields.
Note that, with such regularity assumption on $b$, equation \eqref{eq:FP1} can be written as
\begin{equation*}
\partial_{t} \rho - \sigma \Delta_{\GG} \rho - \diver_{\GG} \big(b(t, x)\big)\rho-b(t, x)\nabla_{\GG}\rho=0 \quad (t, x) \in (0,T) \times \R^{d}.
\end{equation*}

We show the existence of a global classical solution to \eqref{eq:FP1} which, besides its own interests, it will be a fundamental tool to study the Hamilton-Jacobi equation looking at \eqref{eq:FP1} as the dual of such equation for a suitable choice of the drift $b$. It will also play a crucial role in the study of MFG. 
In this case  the model we have in mind is obtained taking 
\begin{equation*}
b(t,x)= \gamma |\nabla_{\GG} u|^{\gamma-2}\nabla_{\GG}u
\end{equation*}
where  $\gamma\geq 2$, $u$ is the solution to a Hamilton-Jacobi equation and $\nabla_{\GG}$ is the horizontal gradient.\\
We also find uniqueness of the solution of the
 Fokker-Planck equation 
by showing uniqueness of the solution to a general linear equation with bounded coefficients. In our opinion this result has its own interest. Let us remark that in \cite{Cinti_2009} a similar result is obtained with a different approach under stronger assumptions on the coefficients and for a different linear subelliptic equation.
\\
 Moreover we also find an H\"older regularity result w.r.t. the flat Wasserstein metric $d_0$ which will be defined in subsection \ref{holdersub}.

Throughout this paper the space $C^{1,\nu}_{\GG}(\R^d)$, defined in Section \ref{subell}, is the H\"older space w.r.t. the horizontal derivatives and the metric associated to the group.\\
\begin{theorem}\label{thm:ex_FP1}
Consider a nonnegative continuous function $\rho_0\in L^\infty(\R^d)$ such that $\int_{\R^d}\rho_0 dx=1$ and  for some $\delta\in(0,1]$
a drift function 
$$b \in C([0,T]; (C^{1,\delta}_{\GG}(\R^{d}))^m)\ {\text{with}} \
\sup_{t\in[0,T]}\|b(t,\cdot)\|_{(C^{1,\delta}_{\GG}(\R^{d}))^m}<+\infty.$$ 
Then, equation~\eqref{eq:FP1} has a unique classical bounded solution~$\rho \in C((0,T); C^{2,\delta}_{\GG, loc}(\R^{d}))$.
Moreover, the following hold. 
\begin{itemize}
\item[($i$)] $0\leq \rho\leq \|\rho_0\|_{L^\infty(\R^d)}$ and $\int_{\R^d}\rho(t, x)\, dx=1$ for every $t\in [0,T]$.
\item[($ii$)]  For every $\varphi\in C([0,T];W^{1,2}_{\GG}(\R^d))$ with $\partial_t\varphi\in L^2(0,T;(W^{1,2}_{\GG}(\R^d))')$ we have
\begin{multline*}
\int_{\R^d}\rho(x,t)\varphi(x,t)\, dx- \int_{\R^d}\rho_0(x)\varphi(x,0)\, dx -
\iint_{(0,t)\times \R^d}\partial_t\varphi (t, x)\rho (t, x) dxdt \\ +\iint_{(0,t)\times \R^d}\nabla_{\GG}\varphi (t, x)\left(\sigma \nabla_{\GG}\rho (t, x)+b (t, x)\rho\right)\, dxdt =0.
\end{multline*}
\item[($iii$)] There exists a constant $C_\rho \geq 0$ such that 
\begin{equation*}
d_0 (\rho_t, \rho_s) \leq C_{\rho}|t-s|^{\frac{1}{2}} \qquad \forall t, s \in [0,T]
\end{equation*}
where $\rho_t$ is the measure associated to the density $\rho(t, \cdot)$.
\end{itemize}
\end{theorem}

The existence part of the statement is obtained approximating problem~\eqref{eq:FP1} by truncation and exploiting some regularity results which strongly rely on the H\"ormander condition. The uniqueness part of the statement is obtained finding a suitable subsolution used as a test function related to the fundamental solution of the horizontal heat equation.

Note that, as proved in \cite{Bell_2004} and \cite{Paoli_2016}, the measure $\{\rho_t\}_{t \in [0,T]}$ represents the law of the stochastic process
\begin{equation*}
d\xi_t = \sum_{j=1}^{m} X_j(\xi_t) \circ dB^{j}_{t} + X_0 (\xi_t)\ dt
\end{equation*}
where $\circ$ denotes the Stratonovich stochastic integral, $\{B^j\}_{j=1, \dots, m}$ are $d$ dimensional independent Brownian motions and $X_0$ is the drift term explicitly provided in \cite{Paoli_2016}. However, since the data can grow more than linearly, we are not able to provide a bound on the moments of such a process which is a key result to get the solution in the non degenerate case.

\subsubsection{Hamilton-Jacobi equation}
\label{HJ}
We are interested in studying the well-posedness of the following Hamilton-Jacobi equation
\begin{equation}\label{eq:HJ}
\begin{cases}
\partial_{t} u(t, x) - \sigma \Delta_{\GG} u(t, x) + |\nabla_{\GG} u (t, x)|^{\gamma} = F(t, x), & (t, x) \in (0,T] \times \R^{d} 
\\
u(0, x)=u_{0}(x), & x \in \R^{d},
\end{cases}
\end{equation}
with $\gamma \geq 2$ under the assumptions
\begin{align}
&F \in C^{0}([0,T]; C^{1}_{\GG}(\R^{d}))\cap L^2((0,T) \times \R^{d}) {\text{ with }}
\sup_{t\in[0,T]}\|F(t,\cdot)\|_{C^{1}_{\GG}(\R^{d})}<+\infty \label{assFu_0}\\
&u_{0} \in W^{2, \infty}_{\GG}(\R^{d}) \cap L^1(\R^{d}),\  u_0\geq 0.\label{u_0}
\end{align} 

To do this we will study a general semi-linear parabolic equation of the following type
\begin{equation}\label{eq:interestedin}
\begin{cases}
\partial_{t} u(t, x) - \sigma \Delta_{\GG} u(t, x) = f(t, x, \nabla_{\GG} u(t, x)), & (t, x) \in (0,T) \times \R^{d} 
\\
u(0, x)=u_{0}(x), & x \in \R^{d}. 
\end{cases}
\end{equation}
From the structure of \eqref{eq:HJ} we assume the following on $f: [0,T] \times \R^{d} \times \R^{m} \to \R$. 
\vspace{0.1cm}
\begin{itemize}
\item[{\bf (HP)}] $f(\cdot, x, p) \in C([0,T])$, $f(t, \cdot, p) \in C^{1}(\R^{d})$  ${\text{ with }}
\sup_{t\in[0,T]}\|f(t,\cdot,p)\|_{C^{1}_{\GG}(\R^{d})}<+\infty$
and there exist $C_{f} \geq 0$, $\gamma \geq 2$ such that 
\begin{align*}
 f(t, x, p) \leq\ & C_{f}(1+|p|^{\gamma}), \quad (t, x, p) \in [0,T] \times \R^{d} \times \R^{m},
 \\
 |f(t,x,p) - f(t,x,q)| \leq\ & C_{f}(|p|^{\gamma-1} + |q|^{\gamma-1})|p-q|, \quad t \in [0,T],\,\, x \in \R^{d},\,\, p,q \in \R^{m}.  
\end{align*} 
\item[{\bf (HP')}] there exists $C_{f} \geq 0$ such that, for every $i=1,\dots,m$, $t \in [0,T]$, $x \in \R^{d}$ and $p,q \in \R^{m}$, there hold 
\begin{align*}
&\,|X_i f(t, x, p)| \leq C_{f}(1+|p|^{\gamma}),
\\
&\,  |\partial_{p_i} f(t, x, p)| \leq C_{f}(1+|p|^{\gamma-1}), \\
& \left|X_i\left(f(t,x,p) - f(t,x,q)\right)\right| \leq C_{f}(|p|^{\gamma-1} + |q|^{\gamma-1})|p-q|,\\
&\left|\partial_{p_i} \left(f(t, x, p)-f(t, x, q)\right)\right| \leq C_{f}(|p|^{\gamma-2} + |q|^{\gamma-2})|p-q|.
\end{align*} 
\end{itemize}

The first result is the existence and uniqueness of a classical solution of \eqref{eq:interestedin} for small times.
\begin{theorem}\label{thm:existence}
Let $f$ satisfy {\bf (HP)} and {\bf (HP')} and let $u_{0} \in W^{2, \infty}_{\GG}(\R^{d})$. Then, there exists $T_{0} > 0$ such that for any $T \leq T_{0}$ there exists a unique classical solution $u \in C([0,T], W^{2,\infty}_\GG(\R^d))\cap C^{1 + \frac{\alpha}{2}, 2+\alpha}_{\GG, {\textrm{loc}}}((0,T] \times \R^{d})$ to \eqref{eq:interestedin} such that 
\[
\sup_{t \in [0,T]} \| u(t)\|_{W^{2, \infty}_{\GG}(\R^{d})} \leq \kappa(T_{0})
\]
for some $\kappa(T_{0}) > 0$. 
\end{theorem}

In order to prove the above small-time existence of classical solutions to the semi-linear Hamilton-Jacobi equation \eqref{eq:interestedin} we use the Duhamel  formula and the decay estimate of the subelliptic heat semigroup (\Cref{decayest}). 

Next, we state the second main result on the global in time existence of classical solutions to the Hamilton-Jacobi equation \eqref{eq:HJ}. To this end, we establish an estimate of the horizontal gradient of~$u$ using the Bernstein method suitably adapted to sub-Riemannian framework. In order to apply this method, we need an extra technical assumption for a set of right-invariant vector fields.


\begin{theorem}\label{thm:existence1}
Let $F: [0,T] \times \R^d \to \R$ satisfy \eqref{assFu_0} and let $u_{0}$ satisfy \eqref{u_0}. Let $\{ Y_{1}, \dots, Y_{m}\}$ be a set of right-invariant smooth vector fields on $\R^{d}$ and assume that
\begin{equation}\label{hyp:reg_Y}
\| Y_{j} F\|_{L^{\infty}( [0, T] \times \R^{d})} +\| Y_{j} u_0\|_{L^{\infty}(\R^{d})}<\infty.
\end{equation}
Let $T > 0$ be arbitrary. Then, there exists a classical solution $u \in C^{1+\frac{\alpha}{2}, 2+\alpha}_{\GG, \text{loc}}((0,T] \times \R^{d})$ to \eqref{eq:HJ}. 
\end{theorem}

To obtain the global in time existence, we first show that the solution is bounded in space and time--uniformly w.r.t. the second order coefficient $\sigma$--(\Cref{continuousbound}) and that also its horizontal gradient is bounded globally in time and locally in space (\Cref{horizontalbound})--uniformly w.r.t. $\sigma$. The common main technical point of the above results is the use of the duality between Hamilton-Jacobi equation and Fokker-Plank equation: given a suitable choice of the drift function and the corresponding solution to the Fokker-Plank equation we deduce properties for the solution $u$. This duality property have been investigated in several other settings, see for instance \cite{Cirant_2020}, \cite{Evans_2003}, \cite{Gomes_2016} . Moreover, in order to get the estimate on the horizontal gradient we use the Bernstein method 
adapted to the sub-Riemannian framework. More precisely, one cannot apply a vector field generating the Lie group to the equation because, otherwise, one will get some extra terms involving commutators which are difficult to deal with. In order to overcome this issue, we first differentiate the equation by the family of smooth right-invariant vector fields introduced in the statement. Hence, we obtain a local bound for the gradient of the solution on the distribution generated by such right-invariant vector fields. Finally, since the right-invariant distance is locally Lipschitz equivalent to the left-invariant one, we get the desired estimate. 

We conclude this part, observing that the assumption $\gamma \geq 2$ is needed in order to gain regularity on the solution to the Hamilton-Jacobi equation when the initial data is regular enough. In a forthcoming paper, we will address a similar problem investigating the sub-quadratic case with merely local integrability of the solution. In this case, the integrability assumptions of the drift of the transport equation and the duality approach still allow us to obtain classical solutions to the Fokker-Plank equation, to the Hamilton-Jacobi equation and, consequently, to the MFG system. 

\subsubsection{Mean Field Games}
We conclude this work investigating the existence for small times of solutions to the MFG system
\begin{equation}\label{MFG}
\begin{cases}
-\partial_{t} u(t, x) - \sigma \Delta_{\GG} u(t, x) + |\nabla_{\GG} u (t, x)|^{\gamma} = F_{\textrm{\tiny {MFG}}}[\rho] (x), & (t, x) \in (0,T) \times \R^{d} 
\\
\partial_{t} \rho_{t} - \sigma \Delta_{\GG} \rho - \text{div}_{\GG}(\gamma |\nabla_{\GG} u(t, x)|^{\gamma-2} \nabla_{\GG}u(t, x)\rho)=0, & (t, x) \in (0,T) \times \R^{d}
\\
u(T, x) = u_{T}(x), \quad \rho(0,x)=\rho_{0}(x), & x \in \R^d. 
\end{cases}
\end{equation}
Such a system couples the equations we have previously studied, i.e., the Hamilton-Jacobi equation in \eqref{eq:HJ} in which the potential $F_{\textrm{\tiny {MFG}}}$ is a strongly regularizing nonlocal term which depends on the distribution $\rho$ and the Fokker-Planck equation in \eqref{MFG} whose drift is defined by the optimal feedback associated with \eqref{eq:HJ}. In the following, we consider a coupling function $F_{\textrm{\tiny {MFG}}}: \PP(\R^d) \times \R^d \to \R$ and we assume the following:
\begin{itemize}
\item[{\bf (MFG)}] the map $\rho \mapsto F_{\textrm{\tiny {MFG}}}[\rho](x)$ is Lipschitz continuous w.r.t. $d_0$ from $\PP(\R^d)$ to $C^1_{\GG}(\R^d)$ and the function $F_{\textrm{\tiny {MFG}}}[\rho](\cdot)$ is bounded in $C^1_{\GG}(\R^d) \cap L^2(\R^d)$ uniformly w.r.t. $\rho$.
\end{itemize}

MFG theory, introduced in \cite{bib:LL1, bib:LL2, bib:LL3}, is devoted to the study of differential games with a very large number of interacting agents. A typical model is described by a system of PDEs: a backward-in-time Hamilton-Jacobi equation whose solution is the value function for the generic player (and also provides their optimal choices) and a forward-in-time Fokker-Planck equation which describes how the distribution of individuals changes. The two equations are coupled in a way that takes into account both the state of a single agent and how he/she is influenced by the others. The system of PDEs describes a model with a continuum of players which is, clearly, not realistic. However, the solution of such a system is expected to capture the behavior of Nash equilibria for differential N-players game as the number of agents goes to infinity. For an extensive and detailed introduction to the subject we refer to  \cite{bib:BFY}, \cite{bib:NC}, \cite{bib:DEV} and references therein. 

The motivation for studying our model is the following: 
heuristically, given $X_1$, $\dots, X_m$ smooth vector fields on $\R^d$, each single player can move only along the directions generated by the set of vector fields. 
In \cite{Dragoni_2018,Gomes_2020}  stationary subelliptic MFG systems have been studied in the torus.
We also cite the papers \cite{Mannucci_2020, Mannucci_2022, Achdou_2022, bib:CM, Mendico}, where, in the whole space, the generic player has some forbidden directions because it follows either a dynamic generated by the vector fields on Heisenberg group, or a Grushin dynamic or it controls its acceleration.
In particular, the results in \cite{Mannucci_2020, Mannucci_2022} are obtained under the key assumption that the coefficients of the vector fields $X_i$ grow at most linearly.
Up to the authors' knowledge, no results are available for parabolic MFG problem in unbounded Lie groups. Clearly, the unboundedness of the state space, eventually with unbounded vector fields, gives rise to several difficulties to overcome.

From the optimization point of view, a generic player wants to minimize the cost 
\begin{equation*}
\mathbb{E} \left[\int_{t}^{T} (C|\alpha(s)|^{\gamma^{*}} + F_{\textrm{\tiny {MFG}}}[\rho(s)](\xi(s))\ ds + u_T (\xi(T)) \right]
\end{equation*}
where $\gamma^{*}$ is the conjugate index of $\gamma$ and $\xi(\cdot)$ denotes the stochastic process that solves 
\begin{equation*}
d \xi(s) = \alpha(s) ds + \sum_{j=1}^{m} X_j(\xi(s)) \circ dW_j
\end{equation*}
where $\alpha$ is the control chosen by the player while $\circ$ and $W_j$ are as before.\\
We prove a small time existence result of a classical solution. Note that we cannot prove the existence for any time $T$ because of the lack of compactness in the results of Theorem \ref{thm:existence1}.\\
\begin{theorem}\label{MFGmain}
Assume $\gamma \geq 2$. Let $u_{T} \in W^{2, \infty}(\R^{d})$, $\rho_{0} \in L^{\infty}(\R^d) \cap L^1(\R^d)$ such that $\rho_0\geq0$ and $\int_{\R^d} \rho_0(x)\ dx = 1$ and let $F_{\textrm{\tiny {MFG}}}$ satisfies {\bf (MFG)}. Then, there exists $T_0 > 0$ such that for any $T \leq T_0$ there exists a classical solution $(u, \rho) \in C((0,T); C^{2+\alpha, 1+\frac{\alpha}{2}}_{\GG, {\textrm{loc}}}(\R^{d})) \times C((0,T); C^{2,\nu}_{\GG, {\textrm{loc}}}(\R^{d}))$ to \eqref{MFG}. 
\end{theorem}

\section{Sub-Riemannian setting}
%
%

Let $(\GG, *)$ be a homogeneous Lie group and let $\{\delta_{\lambda}\}_{\lambda > 0}$ be a family of dilations which are automorphisms of the group, i.e., 
\begin{equation*}
\delta_\lambda(x*y)=\delta_\lambda(x)*\delta_\lambda(y)
\end{equation*}
for all $x, y \in \GG$ and $\lambda > 0$. Let $\g$ be the Lie algebra associated with the group $\GG$. The dilations of the group induce a direct sum decomposition on $\g$, i.e., 
\begin{equation*}
\g = V_1 \oplus \dots \oplus V_k.
\end{equation*}
In particular, $V_1$ is called the horizontal layer and its elements are left-invariant vector fields. We can identify $(\GG, *)$ with $\R^d$ via the so-called exponential map 
\begin{equation*}
exp: \g \to \GG 
\end{equation*}
which turns out to be a diffeomorphism. Given a basis $X_1$, $\dots, X_d$ adapted to the stratification, any $x \in \GG$ can be written in a unique way as 
\begin{equation*}
x= exp(x_1X_1+\dots + x_dX_d)
\end{equation*}
and one can identify $x$ with $(x_1, \dots, x_d)$ and $\GG$ with $(\R^d, \circ)$ where the group law is given by the Baker-Campbell-Hausdorff formula. 

Hence, hereafter we work on $\R^d$ and we consider an orthonormal basis $X_1, \dots, X_m$ of the horizontal layer $V_1$. 
We assume that such a family of vector fields satisfies the H\"ormander condition, i.e.,
\begin{equation*}
\text{Lie}(X_1, \dots X_m)(x)=\R^d \quad \text{for all}\,\, x \in \R^d. 
\end{equation*}
and to be homogeneous of degree one w.r.t. the family of dilations. 

Moreover, given $a, b \in \R$ we say that an absolutely continuous curve $\gamma: [a,b] \to \R^d$ is horizontal if there exist $\alpha_1, \dots, \alpha_m \in L^1(a, b)$ such that 
\begin{equation*}
\dot \gamma(t) = \sum_{j=1}^{m} \alpha_j(t) X_j(\gamma(t)), \quad \text{a.e.}\,\, t \in [a,b]
\end{equation*} 
and the length of $\gamma$ is defined as
\begin{equation*}
\ell(\gamma) = \int_{a}^{b} |\alpha(t)|\ dt. 
\end{equation*}
Under the H\"ormander condition, a well-known result by Chow states that any two points on $\R^d$ can be connected by an horizontal curve. Hence, the definition of Carnot-Carath\'eodory distance is well-posed
\begin{equation*}
d_{\SR}(x, y) = \inf\{\ell(\gamma): \gamma\,\,\text{is an horizontal curve joining}\,\, x\,\,\text{to}\,\, y\}. 
\end{equation*}
One can prove a variational interpretation of the above distance as
\begin{equation*}
\inf\{T > 0 :\, \exists \gamma:[0,T]\to \R^d, \text{ horizontal and joining}\,\, x\,\, \text{to}\,\, y\,\, \text{with}\,\, |\gamma(t)| \leq 1\}.
\end{equation*}
Note that, the Carnot-Carath\'eodory distance is not equivalent to the Euclidean one. Indeed, it is well-known that, for any $K$ compact set, there exists a constant $C > 0$ such that, for any $y$ and $x$ in $K$, we have
\begin{equation*}
\frac{1}{C}|x-y| \leq d_{\SR}(x,y) \leq C|x-y|^{\sigma(x)}
\end{equation*}
where $\sigma(x) \in \N$ is the nonholonomic degree at $x \in \R^d$, that is, the maximum of the degrees of the iterated brackets occurring to fulfill the H\"ormander condition. 
Using the family of dilations and the sub-Riemannian distance one can define a norm on $\R^d$ tailored from the Lie group, 
\begin{equation*}
\|x\|_{\SR} = d_{\SR}(0,x). 
\end{equation*}
However, from the homogeneity of the vector fields $X_1, \dots, X_m$ and the stratification of $\R^d$, one can define an homogeneous norm $\|\cdot \|_{\GG} $ and the homogeneous dimension $Q$ of the group as
\begin{equation*}
\|x\|_{\GG} = \left(\sum_{j=1}^{k}|x_j|^{\frac{2k!}{j}} \right)^{\frac{1}{2k!}} \quad \text{and} \quad Q =\sum_{j=1}^{k} j\  \text{dim}\ V_j. 
\end{equation*}

\begin{example}\em
Examples of homogeneous Lie groups are Heisenberg-type groups, Engel group and Martinet group (see, for instance, \cite{Montgomery}). 
\end{example}

For completeness, we write the following result on homogenous functions of Lie groups which will be useful later. 

\begin{lemma}\label{lem:homog}
Let $g: \R^d \to \R$ be a homogeneous function of degree 1 and assume 
$|g(x)| \leq c$
for any $\|x \|_{\GG} \leq 1$. Then
$
| g(x)| \leq c\|x\|_{\GG},
$
for all $x \in \R^d$. 
\end{lemma}
\proof 

Take $y \in \R^d$. Then 
\begin{equation*}
|g(y)| = \left|g\left(\delta_{\|y\|_{\GG}} \delta_{\frac{1}{\|y\|_{\GG}}}(y)\right) \right| = \|y\|_{\GG} \left|g\left( \delta_{\frac{1}{\|y\|_{\GG}}}(y)\right)\right|. 
\end{equation*}
So, we conclude, observing that
\begin{equation*}
\| \delta_{\frac{1}{\|y\|_{\GG}}}(y)\|_{\GG} \leq \frac{1}{\|y\|_{\GG}} \|y\|_{\GG} \leq 1. \eqno\square
\end{equation*}

\subsection{Subelliptic equations}\label{subell}

Let us consider a family of smooth left-invariant vector fields $X_1, \dots, X_m$ as in the beginning of this section, satisfying the H\"ormander condition
and a function $u:\R^d \to \R$. We define the horizontal gradient and the horizontal Laplacian of $u$ as
\begin{equation*}
\nabla_{\GG} u= (X_1 u, \dots, X_m u)^{T} \in \R^m
\end{equation*}
and respectively
\begin{equation*}
\Delta_{\GG} u = \sum_{i=1}^{m} X^2_i u \in \R. 
\end{equation*}
For any vector-valued function $u: \R^d \to \R^m$, we will consider the divergence operator induced by the vector fields, that is, 
\begin{equation*}
\diver_{\GG} u = X_1 u_1 + \dots + X_m u_m
\end{equation*}
where $u_i$ denotes the $i$-th component of $u$ for $i=1, \dots, m$. Next, we recall the definition of H\"older space associated with the family of vector fields. For every multi-index $J=(j_1, \dots, j_m) \in \N^m$ let $X^J = X_1^{j_1} \cdots X_m^{j_m}$ and let $|J|=j_1 + \cdots + j_m$ be the length of multi-index.\\
We introduce $C^{0}(\R^d)$ as the set of continuous (possibly unbounded) functions on $\R^d$ and we associate the norm
$\| u\|_{C^{0}(\R^d)}= \| u\|_{L^{\infty}(\R^d)}$.\\
Moreover for $\alpha \in (0,1]$, given $u\in C^{0}(\R^d)$ and $U\subset \R^d$ we define the seminorm 
\begin{equation*}
[u]_{C_{\GG}^{0,\alpha}(U)} = \sup_{\substack{ x, y \in U \\ x \not= y}} \frac{|u(x) - u(y)|}{d_{\SR}(x, y)^{\alpha}},
\end{equation*}
where $d_{\SR}$ is the Carnot-Carath\'eodory distance defined before.
We introduce 
\begin{equation*}
C^{0,\alpha}_{\GG}(\R^d)= \left\{u \in C^0(\R^d): [u]_{C_{\GG}^{0,\alpha}(U)} < \infty, {\textrm{ for every compact }} U\subset \R^d  \right\}
\end{equation*}
and the corresponding norm
$$
\|u\|_{C_{\GG}^{0,\alpha}(\R^d)}:=\| u\|_{C^{0}(\R^d)}+ [u]_{C_{\GG}^{0,\alpha}(\R^d)}.$$
Similarly, for $r \in \N$ and $\alpha \in (0,1]$ we define 
\begin{equation*}
C^{r,\alpha}_{\GG}(\R^d)= \left\{u \in C^0(\R^d): X^{J}u\in C^{0,\alpha}_{\GG}(\R^d), \forall\, |J| \leq r  \right\}
\end{equation*}
and the corresponding norm
$$
\|u\|_{C_{\GG}^{r,\alpha}(\R^d)}:=\sum_{0 \leq |J| \leq r} \left(\| X^{J}u\|_{{C^0}(\R^d)}+ [X^{J}u]_{C^{0,\alpha}_{\GG}(\R^d)}\right).$$
The spaces $C^{r, \alpha}_{\GG}(\R^d)$ with the norm $\|\cdot\|_{C_{\GG}^{r,\alpha}(\R^d)}$
are Banach spaces for any $r \in \N$ and any $\alpha \in (0,1]$.\\
 We conclude this preliminary section recalling, also, the definition of horizontal Sobolev spaces. Let $r \in \N$ and $1 \leq p \leq \infty$. We define the space 
\begin{equation*}
W^{r, p}_{\GG}(\R^d) = \left\{u \in L^p (\R^d) : X^J u \in L^p(\R^d), \,\, \forall\, |J| \leq r  \right\}
\end{equation*}
endowed with the norm 
\begin{equation*}
\| u \|_{W^{r,p}_{\GG}(\R^d)} = \left(\sum_{|J| \leq r} \int_{\R^d} |X^J u(x)|^{p}\ dx\right)^{\frac{1}{p}},
\end{equation*}
for $r\in [1,\infty)$ and similarly for $p=\infty$.

We denote by $C^{r, \alpha}_{\GG, \text{loc}}(\R^d)$ and $W^{r, p}_{\GG, \text{loc}}(\R^d)$ the horizontal H\"older space and horizontal Sobolev space, respectively, as in the above manner with $\R^d$ replaced by any compact subset $\Omega$.

In order to have a self-contained presentation of the work, next we recall the main results used below on linear parabolic subelliptic equations and a Sobolev embedding result. 

\begin{theorem}\cite[Theorem 1]{Lamberton_1987}\label{lamberton}
Let $p \in (1,\infty)$ and $\Omega$ an open subset of $\R^d$. If $f \in L^p([0,T] \times \Omega)$, then the problem
\begin{equation*}
\begin{cases}
\partial_t u - \Delta_{\GG} u = f
\\
u(0, x) = 0
\end{cases}
\end{equation*}
admits a solution $u \in C([0,T]; L^p(\Omega))$ such that $\partial_t u$, $\Delta_{\GG} u \in L^p([0,T]\times\Omega)$ and 
\begin{equation*}
\| \partial_t u\|_{L^p([0,T] \times \Omega)} + \| \Delta_{\GG} u\|_{L^p([0,T] \times \Omega)} \leq C_p \|f\|_{L^p([0,T] \times \Omega)}
\end{equation*}
where $C_p$ depends only on $p$ and on the holomorphic constant of the semi-group $e^{t \Delta_{\GG}}$. 
\end{theorem}

\begin{theorem}\cite[Theorem 1.4]{Frentz_2012}\label{immersion}
Let $\Omega \subset \R^d$ and let $U$ be a compact subset of $(0,T] \times \Omega$. Let $Q + 2 < p < 2(Q+2)$. Then, there exists a positive constant $C$, depending only on $U$, $\Omega$, $T$ and $p$ such that for $\alpha = \frac{1}{p}(p - (Q+2))$ and for every $u \in W^{1,p}([0,T]; W^{2,p}_{\GG}(\Omega))$ we have 
\begin{equation*}
\|u\|_{C^{1, \alpha}_{\GG}(U)} \leq C\left(\|\partial_t u\|_{L^{p}((0,T] \times \Omega)}+\|u\|_{W^{2,p}_{\GG}((0,T] \times \Omega)}\right). 
\end{equation*}
\end{theorem}
(Recall that $Q$ is the homogeneous dimension of the group defined before).

\section{Fokker-Planck equation}

\subsection{Existence of solutions}

We consider the Fokker-Planck equation
\begin{equation}\label{eq:FP}
\begin{cases}
\partial_{t} \rho - \sigma \Delta_{\GG} \rho - \text{div}_{\GG}(b(t, x)\rho)=0, & (t, x) \in (0,T) \times \R^{d}
\\
\rho(0,x)=\rho_{0}(x), & x\in\R^d
\end{cases}
\end{equation}
where $b \in C([0,T]; (C^{1,\delta}_{\GG}(\R^{d}))^m)$, for some $\delta\in(0,1)$.



\begin{proposition}\label{prop:ex_FP}
Under the assumptions of 
\Cref{thm:ex_FP1}
equation~\eqref{eq:FP} has a classical bounded solution~$\rho \in C((0,T); C^{2,\delta}_{\GG, \textrm{loc}}(\R^{d}))$. Moreover,
\begin{equation}\label{L1}
0\leq \rho\leq \|\rho_0\|_{L^\infty(\R^d)}, \quad \text{and} \quad  \int_{\R^d}\rho(t, x)\, dx=1, \,\, \forall\ t\in [0,T].
\end{equation}
\end{proposition}

\proof
Consider the initial-boundary value problem
\begin{equation}\label{eq:P_R}
\left\{\begin{array}{ll}
\partial_{t} \rho_R - \sigma \Delta_{\GG} \rho_R - \text{div}_{\GG}(b(t, x)\rho_R)=0, &\qquad (t, x) \in (0,T) \times B_R\\
\rho_R(0,x)=\rho_{0}(x),&\qquad x\in B_R\\
\rho_R(t,x)=0,&\qquad (t,x)\in (0,T)\times \partial B_R  
\end{array}\right.
\end{equation}
where $B_R$ is the ball of radius $R$ w.r.t. the Carnot-Carath\`eodory distance. We shall first solve problem~\eqref{eq:P_R} establishing several properties of~$\rho_R$ and after, letting $R\to\infty$, we obtain a solution to problem~\eqref{eq:FP} with the desired properties. 

Invoking Lions' Theorem (see \cite[theorem X.9]{Brezis}), we infer that there is a unique function $\rho_R\in L^2(0,T;H^{1}_{0,\GG}(B_R))\cap C([0,T];L^2(B_R))$ with $\partial_t\rho_R\in L^2(0,T;H^{-1}_{\GG}(B_R))$ such that
\begin{equation}\label{ll}
\begin{cases}
\int_{B_R}\partial_{t} \rho_R(t)\,v\, dx + \int_{B_R} \left(\sigma\nabla_{\GG} \rho_R(t)+b(t)\rho_R(t)\right)\nabla_{\GG}v =0, & v\in H^{1}_{0,\GG}(B_R), \textrm{a.e.}\, t\in [0,T]
\\
\rho_R(0,x)=\rho_{0}(x),&\ x\in B_R
\end{cases}
\end{equation}
where $H^{1}_{0,\GG}(B_R)$ is the closure of $C^1_{\GG,0}(B_R)$ in $W^{1,2}(B_R)$ and $H^{-1}_{\GG}(B_R)$ is its dual.
Observe that $\rho_R$ is a distributional solution to 
\begin{equation*}
\partial_{t} \rho_R - \sigma \Delta_{\GG} \rho_R= \left(\text{div}_{\GG} b\right)\rho_R+b\nabla_{\GG}\rho_R
\end{equation*}
where the right hand side belongs to $L^2((0,T)\times B_R)$. 
The results in \cite[Theorem 18]{RS_77} ensures that: $\partial_t\rho_R$, $X_i\rho_R$ and $X_iX_j\rho_R$ belong to $L^2((0,T)\times B_R)$. In particular, we deduce that the differential equation in~\eqref{eq:P_R} is satisfied for a.e. $(t,x)\in (0,T)\times B_R$.\\
We claim that $\rho_R\in C([0,T]\times B_R)$ and that, for every domain $\Omega\subset (0,T)\times B_R$, there exist a constant $K(\Omega,R)$ (depending on the assumptions, on $\Omega$ and on $R$) and a constant $K'(\Omega)$ (depending on the assumptions and on $\Omega$), such that
\begin{equation}\label{eq:claim1}
\|\rho_R\|_{C^{2,\delta}_{\GG}(\Omega)}\leq K(\Omega,R)\quad\textrm{and}\quad
\|\rho_R\|_{C^{2,\delta}_{\GG}(\Omega)}\leq K'(\Omega) \quad\textrm{for $R$ sufficiently large}
\end{equation}
(recall that $\delta$ is the H\"older exponent of $\nabla_{\GG}b$).
Indeed, consider a sequence $\{b_n\}_n$ of drifts such that $b_n\in C^\infty$ and $b_n$ uniformly converges to $b$ in $[0,T]\times B_R$ as $n\to\infty$. Therefore, by the same arguments as before, problem~\eqref{eq:P_R} with $b$ replaced by $b_n$ has a solution $\rho_{R,n}$. Applying iteratively \cite[Theorem 18]{RS_77}, we infer that $\rho_{R,n}\in C^\infty$ and, by standard comparison principle, we get $\|\rho_{R,n}\|_{L^\infty([0,T]\times \R^d)}\leq \|\rho_0\|_{L^\infty(\R^d)}$. Moreover, the results in \cite[Theorem 1.1]{BramantiBrandolini_07} (with $k=0$) ensure that $\rho_{R,n}$ fulfills \eqref{eq:claim1} with  constants $K$ and $K'$, both independent of $n$. Letting $n\to \infty$, we accomplish the proof of our claim~\eqref{eq:claim1}.

Since $\rho_R$ is continuous and $\text{div}_{\GG} b$ is bounded, standard comparison principle entails
\begin{equation*}
0\leq \rho_R(t,x)\leq \|\rho_0\|_{L^\infty(\R^d)}\qquad\forall (t,x)\in[0,T]\times B_R
\end{equation*}
and that, for each $(t,x)$, the value $\rho_R(t,x)$ is nondecreasing with respect to~$R$, i.e., 
\begin{equation*}
\rho_{R_1}(t,x)\geq \rho_{R_2}(t,x) \qquad \forall R_1\geq R_2, (t,x)\in[0,T]\times B_{R_2}.
\end{equation*}
From the properties proved so far we obtain that for each $(t,x)$ there exists the limit $\displaystyle{\lim_{R\to\infty}}\rho_R(t,x)$ which we denote by $\rho(t,x)$ and, clearly, $0\leq \rho\leq \|\rho_0\|_{L^\infty(\R^d)}$. By a standard diagonalization process, using~\eqref{eq:claim1}, $\rho\in C^{2,\delta}_{\GG, \text{loc}}(\Omega)$ for every domain $\Omega\subset (0,T)\times \R^d$.

We now proceed with the proof of the second part of \eqref{L1}. First, we consider again the approximating problem \eqref{eq:P_R}.
We integrate the equation on $[0,T]\times B_R$, we use the divergence theorem and we note that 
$\partial\rho_R/\partial \nu\leq 0$ where $\nu$ is the outward pointing normal to $\partial B_R$. Hence 
we get that $\rho_R(t) \in L^1(B_R)$ independently of $R \geq 0$ for any $t \in (0, T]$. Moreover, letting $R \uparrow \infty$ we deduce $\rho(t) \in L^1(\R^d)$ for any $t \in (0,T]$. Let us now consider a function $\xi \in C^{\infty}_{c}(\R^d)$ such that $\xi(x) = 1$ for any $x \in B_1$ and $\xi(x) = 0$ for any $x \in \R^d \backslash B_2$, and define $\xi_R(x) = \xi(\frac{x}{R})$ for each $R > 0$. Hence, multiplying \eqref{eq:FP} by $\xi_R$ and integrating by parts we have
\begin{equation*}
\int_{\R^d} \rho(t, x)\xi_R(x)\ dx + \iint_{(0, t) \times \R^d} \big(-\sigma\Delta_{\GG} \xi_R(x) + b \nabla\xi_{R}(x) \big)\rho(t, x)\ dtdx = \int_{\R^d} \rho_0(x)\xi_R (x)\ dx. 
\end{equation*}
Observing that 
\begin{equation*}
\Delta_{\GG} \xi_R(x) = \frac{1}{R^2}\Delta_{\GG} \xi(\frac{x}{R}) \quad \text{and} \quad \nabla_{\GG}\xi_R(x) = \frac{1}{R} \nabla_{\GG}\xi(\frac{x}{R}),
\end{equation*}
by dominated convergence theorem as $R \uparrow \infty$, we conclude
\begin{equation*}
\int_{\R^d} \rho(t, x)\ dx = \int_{\R^d} \rho_0(x)\ dx. \eqno\square
\end{equation*}

\begin{proposition}\label{prop:ex_FP1}
Under the assumptions of 
\Cref{thm:ex_FP1},
 let $\rho$ be the solution of \eqref{eq:FP} constructed in \Cref{prop:ex_FP}. 
Then:
\begin{itemize}
\item[($i$)] there exists a constant $K$ depending on $b$, $\sigma$ and $\|\rho_0\|_{L^\infty(\R^d)}$ such that
\begin{align}\label{eq:step7}
\begin{split}
\int_{\R^d}|\rho(t,x)|^2\, dx\leq\ & K\|\rho_0\|_{L^2(\R^d)}^2\qquad\forall t\in[0,T],
\\
\int_{[0,T]\times \R^d}|\nabla_{\GG}\rho(t,x)|^2\, dxdt \leq\ & K\|\rho_0\|_{L^2(\R^d)}^2.
\end{split}
\end{align}
\item[($ii$)] For every $\varphi\in C([0,T];W^{1,2}_{\GG}(\R^d))$ with $\partial_t\varphi\in L^2(0,T;(W^{1,2}_{\GG}(\R^d))')$, we have
\begin{multline}\label{FPweak}
\int_{\R^d}\rho(x,t)\varphi(x,t)\, dx- \int_{\R^d}\rho_0(x)\varphi(x,t)\, dx +\iint_{(0,t)\times \R^d}\nabla_{\GG}\varphi(x,t)\left(\sigma \nabla_{\GG}\rho(x,t)+b\rho(x,t)\right)\, dxdt\\-
\iint_{(0,t)\times \R^d}\partial_t\varphi(x,t)\rho(x,t) dxdt=0.
\end{multline}
\end{itemize}
\end{proposition}

\proof 
For $R>0$, consider the solution~$\rho_R$ to \eqref{eq:P_R} found before.
For simplicity of notation, we shall denote by $K$ a constant which may change from line to line but which always depends only on the assumptions (in particular it is independent of $R$).
Assume for the moment that, for every $t\in[0,T]$ and $R>0$, there holds
\begin{equation}\label{eq:claim6}
\frac{d}{dt}\left(\int_{B_R}|\rho_R(t,x)|^2\, dx\right)+\sigma\int_{B_R}|\nabla_{\GG}\rho_R(t,x)|^2\, dx \leq \frac{\|b\|_{L^{\infty}( [0, T] \times \R^{d})}^2}{\sigma}\int_{B_R}|\rho_R(t,x)|^2\, dx.
\end{equation}
From \eqref{eq:claim6}, we deduce
\begin{equation*}
\frac{d}{dt}\left(\int_{B_R}|\rho_R(t,x)|^2\, dx\right) \leq \frac{\|b\|_{L^{\infty}}^2}{2\sigma}\int_{B_R}|\rho_R(t,x)|^2\, dx.
\end{equation*}
Hence, by Gromwall's lemma, we infer
\begin{equation}\label{eq:stimaGrom}
\int_{B_R}|\rho_R(t,x)|^2\, dx\leq e^{\frac{\|b\|_{L^{\infty}}^2}{2\sigma}}\int_{B_R} \rho_0^2(x)\, dx\leq e^{\frac{\|b\|_{L^{\infty}}^2}{2\sigma}}\|\rho_0\|_2^2 \qquad \forall t\in[0,T], \, R>0.
\end{equation}
So, as $R \uparrow \infty$  we obtain the former estimate in~\eqref{eq:step7}. On the other hand, integrating~\eqref{eq:claim6}, we have
\begin{multline*}
\int_{B_R}|\rho_R(T,x)|^2\, dx - \int_{B_R}|\rho_0(x)|^2\, dx +\frac{\sigma}{2}\int_{[0,T]\times B_R}|\nabla_{\GG}\rho_R(t,x)|^2\, dxdt\\ \leq \frac{\|b\|_{L^{\infty}}^2}{2\sigma}\int_{[0,T]\times B_R}|\rho_R(t,x)|^2\, dxdt\leq e^{\frac{\|b\|_{L^{\infty}}^2}{2\sigma}}T\|\rho_0\|_{L^2(\R^d)}^2
\end{multline*}
where the last inequality is due to relation~\eqref{eq:stimaGrom}.
Again by~\eqref{eq:stimaGrom}, we deduce that there exists a constant $K$ (independent of $R$ and $\rho_0$) such that
\begin{equation*}
\int_{[0,T]\times B_R}|\nabla_{\GG}\rho_R(t,x)|^2\, dxdt \leq K\|\rho_0\|_{L^2(\R^d)}^2.
\end{equation*}
As $R \uparrow \infty$, we obtain the latter estimate in~\eqref{eq:step7}. It remains to prove estimate~\eqref{eq:claim6}. To this end,
using $\rho_R$ as test function for~\eqref{eq:P_R}, we get
\begin{equation*}
\frac{d}{dt}\left(\int_{B_R}|\rho_R(t,x)|^2\, dx\right) + 2\sigma \int_{B_R}|\nabla_{\GG}\rho_R(x,t)|^2\, dx +2\int_{B_R} b(t,x)\rho_R(x,t)\nabla_{\GG}\rho_R(x,t)\, dx=0.
\end{equation*}
Using H\"older inequality on the last term, we get~\eqref{eq:claim6}. Finally, from the same reasoning we also get \eqref{FPweak}. \qed

\subsection{Uniqueness and regularity}

We recover the uniqueness of solutions to \eqref{prop:ex_FP} by showing the uniqueness of the classical solution to the general linear equations with bounded coefficients of the form
\begin{equation}\label{eq:GFP}
\begin{cases}
\partial_t \rho - \sigma \Delta_{\GG} \rho + B \cdot \nabla_{\GG} \rho + Q \rho = 0 & (t, x) \in (0,T) \times \R^d
\\
\rho(0,x) =\rho_0 & x \in \R^d.  
\end{cases}
\end{equation}
Let us remark that in \cite{Cinti_2009} a similar result is obtained with a different approach under stronger assumptions on the coefficients and for a different linear subelliptic equation.

\begin{proposition}\label{prop:uniqueness}
Let $\rho_0 \in L^\infty(\R^d)$. 
Let $B$ and $Q$ be bounded continuous functions on $[0,T] \times \R^d$ and, moreover, assume that $B$ has a continuous and bounded horizontal gradient. For $j=1, 2$, let $\rho_j \in C((0,T); C^{2,\nu}_{\GG, \text{loc}}(\R^{d}))$ be two classical solutions to~\eqref{eq:GFP} such that for some positive constant $\beta$ we have 
\begin{equation*}
\int_{0}^{T} \int_{\R^d} |\rho_j (t, x)|e^{-\beta (\|x\|^{2}_{\GG} +1)}\ dxdt < \infty. 
\end{equation*}
Then, $\rho_1 = \rho_2$. 
\end{proposition}
The proof of the proposition is postponed after the following technical lemma.
\begin{lemma}\label{lem:lemma}
Let $\rho_j$, for $j=1, 2$, be two solutions to \eqref{eq:GFP} such that for some $\beta > 0$ we have
\begin{equation}\label{eq:assumption}
\int_{0}^{T} \int_{\R^d} |\rho_j (t, x)|e^{-\beta (\|x\|^{2}_{\GG} +1)}\ dxdt < \infty
\end{equation}
and let 
\begin{equation*}
\tau_0 = \inf\{t \in [0,T] : \rho_1(\cdot, t) \not = \rho_2(\cdot, t)\} \in [0, T). 
\end{equation*}
For $\beta_1 > \beta$ 
and $\bar\beta > 0$, define the function 
\begin{equation*}
\Phi(t, x) = e^{- (\beta_1 + \bar\beta(t-\tau_0))(\|x\|_{\GG}^{2} +1)}. 
\end{equation*}
Then, 
\begin{equation}\label{eq:0}
\partial_t \Phi + \sigma \Delta_{\GG} \Phi + B \cdot \nabla_{\GG} \Phi + (\text{div}_{\GG}B) \Phi \leq 0, \quad (t, x) \in (\tau_0, \tau) \times \R^d 
\end{equation}
for $\tau \in (\tau_0, T]$, with $\tau-\tau_0$ sufficiently small and $\bar\beta$ sufficiently large, and
\begin{equation}\label{eq:1}
\int_{\tau_0}^{\tau} \int_{\R^d} |\rho_j (t, x)| \Phi(t, x)\ dtdx < \infty \quad \text{and} \quad \int_{\tau_0}^{\tau} \int_{\R^d} |\rho_j (t, x) \nabla_{\GG}\Phi(t, x)|\ dtdx < \infty.
\end{equation}
\end{lemma}
\proof 

From the homogeneity of the norm $\| \cdot\|_{\GG}$  and using \Cref{lem:homog} we have that 
\begin{equation*}
|X_j (\|x\|_{\GG}^{2})|^{2} \leq C \|x\|_{\GG}^{2}, \quad  |\nabla_{\GG} (\|x\|_{\GG}^{2})|^{2} \leq C \|x\|_{\GG}^{2} \quad \text{and}\quad |\Delta_{\GG}(\|x\|_{\GG}^{2})|^{2} \leq C
\end{equation*}
for a suitable constant $C \geq 0$. Hence, setting for simplicity $\bar\beta_1 = \beta_1 + \bar\beta(t-\tau_0)$ we have 
\begin{align*}
& \partial_t \Phi + \sigma \Delta_{\GG} \Phi + B \cdot \nabla_{\GG} \Phi + (\text{div}_{\GG}B) \Phi 
\\
=\ & (-\bar\beta(\|x\|_{\GG}^{2} + 1) + \sigma \bar\beta_{1}^{2} |\nabla_{\GG}(\|x\|_{\GG}^{2})|^{2} - \sigma \bar\beta_1 \Delta_{\GG}(\|x\|_{\GG}^{2}) - \bar\beta_1 B \cdot \nabla_{\GG}(\|x\|_{\GG}^{2}) + \text{div}_{\GG}B) \Phi
\\
\leq\ & \left(-\bar\beta(\|x\|_{\GG}^{2} + 1) + \sigma \bar\beta_{1}^{2} C\|x\|_{\GG}^{2} + \sigma \bar\beta_1 C + \bar\beta_1 \|B\|_{L^{\infty}} C\|x\|_{\GG} + \|\text{div}_{\GG}B\|_{L^{\infty}}\right) \Phi.
\end{align*}
The proof of \eqref{eq:0} is thus complete by choosing $\tau-\tau_0$ sufficiently small and $\bar\beta$ sufficiently large. Bounds \eqref{eq:1} are an easy consequence of the choice of $\beta_1$ and assumption \eqref{eq:assumption}. \qed

\medskip
\noindent{\it Proof of \Cref{prop:uniqueness}.} 
Without any loss of generality, we assume $Q \geq 0$ and, we proceed by contradiction assuming that $\rho_1 \not = \rho_2$. Set
\begin{equation*}
\tau_0 = \inf \{t \in [0,T]: \rho_1(\cdot, t) \not= \rho_2(\cdot, t)\}. 
\end{equation*}
The continuity of $\rho_j$, for $j=1, 2$, ensure that the function $\rho = \rho_1 - \rho_2$ satisfies
\begin{equation}\label{eq:48bis}
\begin{cases}
\partial_t \rho - \sigma \Delta_{\GG} \rho + B \cdot \nabla_{\GG} \rho + Q \rho = 0, & (t, x) \in (\tau_0, T) \times \R^d
\\
\rho(\tau_0, x) = 0, & x \in \R^d. 
\end{cases}
\end{equation}
For any $\eps > 0$ define the function 
\begin{equation*}
w(t, x) = \sqrt{\rho(t, x)^2 + \eps}
\end{equation*}
and observe that the following equalities hold
\begin{align*}
& \partial_t w = \frac{1}{w} \rho\partial_t \rho
& X_j w = \frac{1}{w}\rho X_j \rho
\\
& X_{j}^{2} w = \frac{\eps}{w^3} (X_j \rho)^2 + \frac{\rho}{w}X_{j}^{2} \rho
& \Delta_{\GG} w = \frac{\eps}{w^3}|\nabla_{\GG}\rho|^2 + \frac{\rho}{w} \Delta_{\GG}\rho. 
\end{align*}
Therefore, multiplying \eqref{eq:48bis} by $\frac{\rho}{w}$ we get
\begin{equation*}
\partial_t w = \sigma \Delta_{\GG} w - \sigma \frac{\eps}{w^3}|\nabla_{\GG}\rho|^2 - B \cdot \nabla_{\GG} w - Q\frac{\rho^2}{w} \leq \sigma \Delta_{\GG} w - B \cdot \nabla_{\GG}w, \quad (t, x) \in (\tau_0, T] \times \R^d.
\end{equation*}
So, for any nonnegative test function $v \in C^{\infty}([\tau_0, T] \times \R^d)$ with bounded support in space and for any $t \in [\tau_0, T]$ there holds
\begin{multline*}
\int_{\R^d} w(t, x) v(t, x)\ dx - \int_{\R^d} w(\tau_0, x) v(\tau_0, x)\ dx 
\\
\leq \int_{\tau_0}^{t} \int_{\R^d} w(s, x) (\partial_t v(s, x) + \sigma \Delta_{\GG} v(s, x) + \text{div}_{\GG}(v(s, x)B))\ dxds.
\end{multline*}
Since $w( \tau_0, \cdot)= \eps$, letting $\eps \downarrow 0$ we deduce 
\begin{equation*}
\int_{\R^d} |\rho(t, x)| v(t, x)\ dx \leq \int_{\tau_0}^{t} \int_{\R^d} |\rho(t, x)|\left|\partial_t v + \sigma \Delta_{\GG}v + \text{div}_{\GG} (v(t, x) B)\right|\ dsdx. 
\end{equation*}
Choose $t \in [\tau_0, \tau]$ and $v=\xi_R \Phi$ where $\tau$ and $\Phi$ are respectively the constant and the function introduced in \Cref{lem:lemma} and $\xi_R \in C^{\infty}(\R^d)$ is a cut-off function such that 
\begin{equation*}
\xi_R(x) = 1, \quad \text{if}\,\, |x| \leq R, \quad \xi_R(x) = 0 , \quad \text{if} \,\, |x| \geq R +1 \quad \text{and}\,\, \|D\xi_R\|_{L^\infty} + \|D^2 \xi_R\|_{L^\infty} \leq 2. 
\end{equation*}
Hence, we get
\begin{multline*}
\int_{\R^d} |\rho(t, x)| \xi_R(x)\Phi(t, x)\ dx
\\
\leq \int_{\tau_0}^{t} \int_{B_{R+1} \backslash B_R} |\rho(s, x)|\Big|(\sigma\Delta_{\GG} \xi_R(x) + B \cdot \nabla_{\GG}\xi_R(x))\Phi(s, x) + 2 \sigma \nabla_{\GG}\xi_R(x) \cdot \nabla_{\GG} \Phi(s, x)\Big|\ dsdx. 
\end{multline*}
Letting $R \uparrow \infty$, by dominated convergence theorem and \Cref{lem:lemma} we have that the right hand side converges to zero and we obtain 
\begin{equation*}
\int_{\R^d} |\rho(t, x)| \Phi(t, x)\ dx \leq 0, \quad \forall\ t \in [\tau_0, \tau]
\end{equation*}
 which entails $\rho(t, x) = 0$ in $(\tau_0, \tau) \times \R^d$ contradicting the definition of $\tau_0$. \qed
 \vskip.4truecm
By \Cref{prop:uniqueness}, the classical solution constructed in \Cref{prop:ex_FP} is unique, hence we proved the following corollary.
\begin{corollary}\label{uniqueness}
Under the assumptions of 
\Cref{thm:ex_FP1},
there exists a unique bounded classical solution $\rho \in C((0,T); C^{2,\nu}_{\GG, \textrm{loc}}(\R^{d}))$ of \eqref{eq:FP}. 
\end{corollary}
%
%

\subsubsection{H\"older regularity and flat metric}\label{holdersub}
Next, we prove H\"older regularity of the solution $\rho$ w.r.t. the so-called flat metric $d_0$ distance here defined.
There are many ways to metrize weak convergence of measures and the one we use here is the following (see, for instance, \cite{bib:CV}): the bounded Lipschitz distance, also called Fortet-Mourier distance of flat Wasserstein metric
\begin{equation*}
d_{0}(m ,m') = \sup\left\{\int_{\R^d} f(x)\,dm(x)-\int_{\R^d} f(x)\,dm'(x) : f : \R^d \rightarrow \R \,\, \text{s.t.}\,\,  \|f\|_{C^{0,1}_{\GG}(\R^d)} \leq 1 \right\}. 
\end{equation*}

\begin{proposition}\label{holder}
Let $\rho$ be the unique solution to~\eqref{eq:FP} constructed in \Cref{prop:ex_FP}. 
Then, there exists $C_{\rho} \geq 0$ such that
\begin{equation*}
d_0 (\rho_t, \rho_s) \leq C_{\rho} |t-s|^{\frac{1}{2}} \quad \forall t, s \in [0,T]. 
\end{equation*}
\end{proposition}

\proof
We argue adapting some ideas of~\cite[Proposition 6.6]{Ersland_2021}. We 
consider the smooth function
\begin{equation*}
\xi(x)=\left\{\begin{array}{ll}
\exp\left\{\frac{1}{\|x\|_\GG^{2k!}-1}\right\}&\quad\textrm{if }\|x\|_\GG\leq 1\\0&\quad\textrm{otherwise.}
\end{array}\right.
\end{equation*}

For $\eps > 0$, let
\begin{equation*}
\xi^{\eps}(x) = \frac{C}{\eps^Q}\xi\left(\delta_{\eps^{-1}}x\right) \qquad (x \in \R^d)
\end{equation*}
be a smooth mollifier with support in $B(0,\eps)$ and where the constant $C$ is independent of~$\eps$ and such that $\int \xi^{\eps}dx=1$ (we recall that $\delta_{\eps^{-1}}(x)$ denotes the dilation of radius $\eps^{-1}$). Note that, by homogeneity of the norm $\| \cdot \|_{\GG}$ we have
\begin{equation}\label{hom}
X_j \xi^{\eps}(x) = \frac{1}{\eps}X_j \xi (x). 
\end{equation}
Let $\varphi $ be a real valued function with $\|\varphi\|_{C^{0,1}_{\GG}(\R^d)}\leq 1$  
and let $\varphi_{\eps}(x) = \xi^{\eps} \star \varphi(x)$ where the symbol ``$\star$'' denotes the convolution based on the operation of the group. Note that, by standard calculus and Lagrange theorem (see~\cite[Theorem 20.3.1]{Bonfiglioli}, there holds:
\begin{equation*}
\|\varphi-\varphi_{\eps}\|_{L^\infty}\leq \eps.
\end{equation*}
Then, 
\begin{equation*}
\int_{\R^d} \varphi_{\eps}(x) (\rho(t, x) - \rho(s, x))\ dx = \int_{s}^{t} \int_{\R^d} (\Delta_{\GG} \varphi_{\eps}(x) - b \cdot \nabla_{\GG} \varphi_{\eps}(x))\ \rho(z, x)\ dzdx. 
\end{equation*}
First, from \eqref{hom} and standard calculus, we obtain
\begin{equation*}
\|\Delta_{\GG} \varphi_{\eps}\|_{C^0(\R^d)} \leq \frac{1}{\eps} \| \varphi\|_{C^{0,1}_{\GG}(\R^d)}. 
\end{equation*}
Hence, 
\begin{equation*}
\int_{\R^d} \varphi_{\eps}(x) (\rho(t, x) - \rho(s, x))\ dx \leq 2C \frac{1}{\eps}(1+\|b\|_{L^\infty})\|\varphi\|_{C^{0,1}_{\GG}(\R^d)}|t-s|
\end{equation*}
which yields to 
\begin{align*}
 \int_{\R^d} \varphi(x) (\rho(t, x) - \rho(s, x))\ dx 
\leq\ & \int_{\R^d} \varphi_{\eps}(x) (\rho(t, x) - \rho(s, x))\ dx + 2 \| \varphi - \varphi_{\eps}\|_{L^\infty}
\\
\leq\ & C\left(\frac{1}{\eps}|t-s| + \eps \right).
\end{align*}
In conclusion, minimizing over $\eps >0$, we get
\begin{equation*}
d_0 (\rho_t, \rho_s) \leq C_{\rho}|t-s|^{\frac{1}{2}} \quad 0 \leq s \leq t \leq T. \eqno\square
\end{equation*}


\vspace{0.5cm}
\noindent{\it Proof of \Cref{thm:ex_FP1}.} From the above analysis we have that existence, uniqueness, $(i)$ and ($ii$) follow from  \Cref{prop:ex_FP}, \Cref{prop:ex_FP1} and \Cref{prop:uniqueness}. Finally, ($iii$) is proved in \Cref{holder}. \qed

\section{Hamilton-Jacobi equation}
\label{sec:HJB}

\subsection{Small-time existence of solutions}

\noindent{\it Throughout this section we assume that assumptions {\bf (HP)} and {\bf (HP' )} are in force and we study equation~\eqref{eq:interestedin}. }

%
We introduce, for simplicity of notation, the space
\begin{equation}\label{CHI}
\X(T)=C([0,T]; W^{2, \infty}_{\GG}(\R^{d}))
\end{equation}
equipped with the norm
\begin{equation*}
\|\varphi\|_{\X(T)} = \sup_{t \in [0,T]}\|\varphi(t)\|_{W^{2, \infty}_{\GG}(\R^{d})}.
\end{equation*}

Before, for proving the existence of a small-time solution to \eqref{eq:interestedin} we need the following decay estimate for the heat semi-group~$e^{t\Delta_{\GG}}$ generated by the horizontal Laplacian. 

\begin{lemma}\label{decayest}
For any $t \in [0,T]$ and any $\varphi \in L^{\infty}(\R^{d})$ we have that 
\begin{equation}\label{eq:decay1}
\| e^{t\Delta_{\GG}}\varphi\|_{L^\infty(\R^{d})} \leq \| \varphi\|_{L^\infty(\R^{d})}, \qquad \| X_i e^{t\Delta_{\GG}}\varphi\|_{L^\infty(\R^{d})} \leq c(T)t^{-\frac{1}{2}}\| \varphi\|_{L^\infty(\R^{d})},
\end{equation}
where $c(T)$ is a constant depending only on $T$.
\end{lemma}
\proof
Let $\Gamma$ be the fundamental solution to the heat operator $(\partial_t-\Delta_{\GG})$, found in~\cite{Bonfiglioli_2003}. 
Following \cite[Theorem 1.2]{Bonfiglioli_2003} and by construction of the heat semigroup we get 
\begin{equation*}
\| e^{t\Delta_{\GG}}\varphi\|_{L^\infty(\R^{d})}  = \| \Gamma (t)\star \varphi \|_{L^\infty(\R^{d})}  \leq \| \varphi\|_{L^\infty(\R^{d})} . 
\end{equation*}
Similarly, still from \cite[Theorem 1.2]{Bonfiglioli_2003} we deduce the second estimate in \eqref{eq:decay1}.
\qed

\vspace{0.5cm}
\noindent{\it Proof of \Cref{thm:existence}.}  Consider $T>0$, which will be chosen later on. For any $k > 0$ we denote by $\X_{k}(T)$ the closed ball of radius $k$ in $\X(T)$, defined in \eqref{CHI}. Given $k > 0$, we consider the map 
\begin{equation*}
\Phi: \X_{k}(T) \to \X_{k}(T)
\end{equation*}
defined by
\begin{equation}\label{fixmap}
\Phi u(t) = e^{t\Delta_{\GG}}u_{0} + \int_{0}^{t}{e^{(t-s)\Delta_{\GG}}f(s, x, \nabla_{\GG}u(s))\ ds}, \quad \forall t \geq 0. 
\end{equation}
Next, we show that there exist $T_{0} > 0$ and $k > 0$  such that 
the map $\Phi$ is well defined, i.e., $\Phi u \in \X_{k}(T)$ for $u \in \X_{k}(T)$ and $\Phi$ is a contraction for any $T\leq T_0$. 

To do so, let us fix $k > 0$ and $u \in \X_{k}(T)$ for some $T > 0$. Then, from \eqref{eq:decay1} and assumptions 
{\bf (HP)}, 
we have that 
\begin{multline}\label{eq:term1}
\left\|  \int_{0}^{t}{e^{(t-s)\Delta_{\GG}}f(s, x, \nabla_{\GG}u(s))\ ds} \right\|_{L^\infty} \leq \int_{0}^{t}{\| f(s, x, \nabla_{\GG}u(s))\|_{L^\infty}\ ds}
\\
\leq \int_{0}^{t}{C_{f}(1+\| \nabla_{\GG} u(s)\|^{\gamma}_{L^\infty})\ ds} \leq C_{f}T(1+k^{\gamma}). 
\end{multline}
Moreover, still from \eqref{eq:decay1} and assumptions 
{\bf (HP)} and {\bf (HP')}, we have that 
\begin{multline}\label{eq:term2}
\left\| \nabla_{\GG}\int_{0}^{t}{e^{(t-s)\Delta_{\GG}}f(s, x, \nabla_{\GG}u(s))\ ds} \right\|_{L^\infty} \leq \int_{0}^{t}{\|\nabla_{\GG} e^{(t-s)\Delta_{\GG}} f(s, x, \nabla_{\GG} u(s))\|_{L^\infty}\ ds} 
\\
\leq c(T)\int_{0}^{t}{C_{f}(t-s)}^{-\frac{1}{2}}(1+\|\nabla_{\GG}u(s)\|^{\gamma}_{L^\infty})\ ds \leq C_{f}2c(T)T^{\frac{1}{2}}(1+k^{\gamma}) 
\end{multline}
 and 
\begin{multline}\label{termi}
\left\| X_iX_j\int_{0}^{t}{e^{(t-s)\Delta_{\GG}}f(s, x, \nabla_{\GG}u(s))}\ ds\right\|_{L^\infty} \leq\int_{0}^{t}\left\| X_ie^{(t-s)\Delta_{\GG}}X_jf(s, \cdot, \nabla_{\GG}u(\cdot,s))\right\|_{L^\infty} \ ds 
\\
\leq c(T)\int_{0}^{t}(t-s)^{-\frac12}\left\|X_jf(s, \cdot, \nabla_{\GG}u(\cdot,s))\right\|_{L^\infty} \ ds
\leq c(T)cT^{\frac12} (1+k^\gamma)
\end{multline}
where $c(T)$ is the constant introduced in~\eqref{eq:decay1} while $c$ is a constant, that depends on $C_f$, $d$ and $m$ (in particular, is independent of~$T$ and $k$) and may change from line to line.
Moreover, by~\eqref{eq:decay1}, standard arguments entail
\begin{align}\label{eq:term0}
\begin{split}
\left\|e^{t\Delta_{\GG}}u_{0}\right\|_{L^\infty(\R^{d})} \leq\ & \|u_{0}\|_{L^\infty(\R^{d})},
\\
\left\|\nabla_{\GG} e^{t\Delta_{\GG}}u_{0}\right\|_{L^\infty(\R^{d})} =\ &\left\|e^{t\Delta_{\GG}}\nabla_{\GG} u_{0}\right\|_{L^\infty(\R^{d})} \leq \|\nabla_{\GG}u_{0}\|_{L^\infty(\R^{d})}, 
\\
\left\| X_iX_j e^{t\Delta_{\GG}}u_0\right\|_{L^\infty(\R^{d})} \leq\ & \left\|e^{t\Delta_{\GG}} X_iX_j u_0\right\|_{L^\infty(\R^{d})} \leq \left\| X_iX_j u_0\right\|_{L^\infty(\R^{d})}. 
\end{split}
\end{align}

By relations~\eqref{eq:term1}, \eqref{eq:term2}, \eqref{termi} and~\eqref{eq:term0}, for $k > 0$ sufficiently large and $T$ sufficently small, there holds
\begin{equation*}
k > C_{f}\|u_0\|_{W^{2, \infty}_{\GG}(\R^{d})}+C_{f}c(1+k+k^{\gamma})T^{1/2}(T^{1/2}+c(T))
\end{equation*}
then, $\Phi u \in \X_{k}(T)$ for any $u \in \X_{k}(T)$.

Next, we proceed to show that $\Phi$ is a contraction. Let $u$, $v \in \X_{k}(T)$. Then, from \eqref{eq:decay1} we have that 
\begin{multline*}
\left\| \int_{0}^{t} e^{(t-s)\Delta_{\GG}} f(s,x, \nabla_{\GG} u(s))\ ds - \int_{0}^{t} e^{(t-s)\Delta_{\GG}} f(s,x, \nabla_{\GG} v(s))\ ds \right\|_{L^\infty}
\\
\leq \int_{0}^{t} \| f(s, x,\nabla_{\GG} u(s)) - f(s, x, \nabla_{\GG} v(s))\|_{L^\infty}\ ds. 
\end{multline*} 
So, by {\bf (HP)} we obtain 
\begin{align}\notag
& \left\| \int_{0}^{t} e^{(t-s)\Delta_{\GG}} f(s,x, \nabla_{\GG} u(s))\ ds - \int_{0}^{t} e^{(t-s)\Delta_{\GG}} f(s,x, \nabla_{\GG} v(s))\ ds \right\|_{L^\infty}
\\ \notag
\leq\ & \int_{0}^{t} C_{f} (\|\nabla_{\GG} u(s)\|_{L^\infty}^{\gamma-1} + \|\nabla_{\GG} v(s)\|_{L^\infty}^{\gamma-1}) \|\nabla_{\GG}u(s) - \nabla_{\GG} v(s)\|_{L^\infty}\ ds 
\\ \label{stima_1}
\leq\ & 2C_{f}Tk^{\gamma-1} \| u-v\|_{\X(T)}. 
\end{align}
By using similar arguments, one gets
\begin{align}\notag
& \left\| \nabla_{\GG} \int_{0}^{t} e^{(t-s)\Delta_{\GG}} f(s,x, \nabla_{\GG} u(s))\ ds -  \nabla_{\GG} \int_{0}^{t} e^{(t-s)\Delta_{\GG}} f(s,x, \nabla_{\GG} v(s))\ ds \right\|_{L^\infty}
\\ \label{stima_2}
\leq\ & 4c(T)C_{f}T^{\frac{1}{2}}k^{\gamma-1} \| u-v\|_{\X(T)}. 
\end{align}
Furthermore, for any $i,j\in\{1,\dots,m\}$, we have
\begin{eqnarray*}
A_{ij}&:=& \left\| X_iX_j \int_{0}^{t} e^{(t-s)\Delta_{\GG}} \left[f(s,x, \nabla_{\GG} u(s)) - f(s,x, \nabla_{\GG} v(s))\right]\ ds \right\|_{L^\infty([0,T]\times \R^d)}\\
&\leq& \int_{0}^{t} \left\| X_i e^{(t-s)\Delta_{\GG}} X_j\left[f(s,\cdot, \nabla_{\GG} u(\cdot,s)) - f(s,\cdot, \nabla_{\GG} v(\cdot,s))\right]\right\|_{L^\infty([0,T]\times \R^d)}
\ ds \\
&\leq& 2T^{1/2}c(T) \left\| X_j\left[f(s,\cdot, \nabla_{\GG} u(\cdot,s)) - f(s,\cdot, \nabla_{\GG} v(\cdot,s))\right]\right\|_{L^\infty([0,T]\times \R^d)}.
\end{eqnarray*}
Moreover, we have
\begin{equation*}\begin{array}{l}
\left\| X_j\left[f(s,\cdot, \nabla_{\GG} u(\cdot,s)) - f(s,\cdot, \nabla_{\GG} v(\cdot,s))\right]\right\|_{L^\infty([0,T]\times \R^d)}\\
\qquad\leq \left\|X_j\left[f(s,\cdot, \nabla_{\GG} u) - f(s,\cdot, \nabla_{\GG} v)\right]\right\|_{L^\infty([0,T]\times \R^d)}\\
\qquad \quad+\sum_{i=1}^{m}\left\|\partial_{p_i}f(s,\cdot, \nabla_{\GG} u(\cdot,s)) - \partial_{p_i}f(s,\cdot, \nabla_{\GG} v(\cdot,s))\right\|_{L^\infty([0,T]\times \R^d)} \left\|X_iX_j u\right\|_{L^\infty([0,T]\times \R^d)}\\
\qquad \quad+\sum_{i=1}^{m}\left\|\partial_{p_i}f(s,\cdot, \nabla_{\GG} v(\cdot,s))\right\|_{L^\infty([0,T]\times \R^d)} \left\|X_iX_j (u-v)\right\|_{L^\infty([0,T]\times \R^d)}.
\end{array}
\end{equation*}
Replacing the last inequality in the previous one, by assumption~{\bf (HP')}, we get
\begin{equation}\label{Astima}
A_{ij}\leq T^{1/2}c(T)c(1+k^{\gamma-1})\|u-v\|_{\X(T)}
\end{equation}
where $c$ is a constant, that depends on $C_f$, $d$ and $m$ (in particular, is independent of~$T$ and $k$) and may change from line to line. Hence, using ~\eqref{stima_1}, \eqref{stima_2} and \eqref{Astima}, we get
\begin{equation*}
\| \Phi u - \Phi v\|_{\X(T)} \leq cC_{f}T^{\frac{1}{2}}(T^{1/2}+2c(T))(1+k^{\gamma-1}) \|u-v\|_{\X(T)}, \quad \forall u, v \in \X_{k}(T)
\end{equation*}
and we conclude choosing $k$ such that 
\begin{equation*}
cC_{f}T^{\frac{1}{2}}(T^{1/2}+2c(T))(1+k^{\gamma-1}) < 1. 
\end{equation*}
Thus, from the fixed point theorem we obtain the existence of a unique solution in $\X_{k}(T)$.
Moreover, from the representation formula provided by the contraction argument, i.e., 
\begin{equation*}
u(t) = e^{t\Delta_{\GG}} u_0 + \int_{0}^{t} e^{(t-s)\Delta_{\GG}}f(s, x, \nabla_{\GG} u(t, x))\ ds
\end{equation*}
and the regularity of the fundamental solution (see \cite[Theorem 1.2]{Bonfiglioli_2003}) we deduce that  $u \in C^{1 + \frac{\alpha}{2}, 2+\alpha}_{\GG}((0,T] \times \R^{d})$. 
\qed


\subsection{Global existence of solutions}

In \Cref{thm:existence} we showed that for a sufficiently small time horizon $T$ there exists a solution to the general semilinear parabolic equation \eqref{eq:interestedin} in $\X(T)$. In this section, we go back considering the Hamilton-Jacobi
\begin{equation}\label{HJJ}
\begin{cases}
\partial_{t} u(t, x) - \sigma \Delta_{\GG} u(t, x) + |\nabla_{\GG} u (t, x)|^{\gamma} = F(t, x), & (t, x) \in (0,T] \times \R^{d} 
\\
u(0, x)=u_{0}(x), & x \in \R^{d}
\end{cases}
\end{equation}
and we prove that there exists a solution for any $T>0$.


In order to prove that such a solution exists for any arbitrary $T > 0$, the key point is the duality feature between the Hamilton-Jacobi equation and the Fokker-Plank equation studied so far. 
For this reason, we first show that the solution $u$ constructed in \Cref{thm:existence} taking $f(t,x,p)=F(t,x)-|p|^{\gamma}$ solves problem \eqref{HJJ} also in a suitable weak (energy) sense and, then, following a standard procedure (see for instance \cite{Cirant_2020}),  we provide the duality relation between the two equations in the sub-Riemannian setting.

%
%
%

\begin{lemma}\label{integrability}
Let $F$ be as in~\eqref{assFu_0}, $u_0$ as in~\eqref{u_0} and let $u$ be the solution of \eqref{HJJ} found in \Cref{thm:existence}.  For any $T\leq T_0$, we have 
\begin{itemize}
\item[(i)] $\nabla_{\GG} u \in L^{p}((0,T)\times\R^d)$ for every  $\gamma \leq p < \infty$;
\item[(ii)] $(\partial_t - \sigma \Delta_{\GG} )\ u \in L^2((0,T)\times\R^d)$;
\item[(iii)] for any $\varphi \in C([0,T]; L^2(\R^d))$ there holds
\begin{multline}\label{HJweak1}
\int_{s}^{\tau}\int_{\R^d}\varphi(t, x)\bigg(\partial_t u(t, x) -\sigma \Delta_{\GG} u(t, x)+ |\nabla_{\GG} u (t, x)|^{\gamma}\bigg) \ dtdx  \\ =  \int_{s}^{\tau}\int_{\R^d} \varphi(t, x)F(t, x)\ dtdx, \quad 0\leq s\leq \tau\leq T\leq T_0.
\end{multline}
\end{itemize}
\end{lemma}
\begin{proof}
Recall that the solution $u$ constructed in \Cref{thm:existence}  belongs to $C([0,T],W^{2,\infty}_\GG(\R^d))$, 
with $T\leq T_0$
and it is given by 
\begin{equation}\label{repreF}
u(t) = e^{t\Delta_{\GG}} u_0 + \int_{0}^{t} e^{(t-s)\Delta_{\GG}}( F(s, x) - |\nabla_{\GG} u(t, x)|^{\gamma})\ ds\qquad \forall t \in [0,T_0]. 
\end{equation}
$(i)$. Following the arguments of~\cite[Theorem B]{ABA} we get that $u(t) \in L^1(\R^d)$:
let $\widetilde{u}$ solve $\partial_t \widetilde{u} - \sigma \Delta_{\GG} \widetilde{u} = 0$ with $\widetilde{u}(0)=u_0$; we refer to \cite[Theorem 1.2]{Bonfiglioli_2003} for the representation formula of~$\widetilde{u}$ and for the regularity of the fundamental solution for the heat equation. By comparison principle, using that $u_0\geq 0$,
we get $0 \leq u\leq \widetilde{u}$ and consequently $u(t) \in L^1(\R^d)$. 
Hence, integrating \eqref{repreF} we deduce $\nabla_{\GG} u \in L^{\gamma}((0,T) \times \R^d)$; Since $\nabla_\GG u \in L^\infty((0,T)\times\R^d)$, by interpolation we conclude $(i)$. \\
$(ii)$. It is an immediate consequence of point~$(i)$ and assumption ~\eqref{assFu_0}.\\
$(iii)$. Clearly, for all $\varphi \in C^{\infty}_{c}([0,T] \times \R^d)$ we have
\begin{multline*}
\int_{s}^{\tau}\int_{\R^d} \varphi(t, x)\bigg(\partial_t u(t, x) -\sigma \Delta_{\GG} u(t, x)+ |\nabla_{\GG} u (t, x)|^{\gamma}\bigg) \ dtdx \\ =  \int_{s}^{\tau}\int_{\R^d} \varphi(t, x)F(t, x)\ dtdx. 
\end{multline*}
By a standard approximation argument, we infer~\eqref{HJweak1}.
\end{proof}
In the following lemma we get a useful relation between the solutions $u$ and $\mu$ respectively of the Hamilton-Jacobi and Fokker-Planck equations using the duality structure of these equations.


\begin{lemma}\label{lem:dualityarg}
Let $u \in \X(T_{0})$ be a solution to \eqref{HJJ} as in \Cref{thm:existence}. Let $\tau \in (0,T_{0}]$. For any $\mu_\tau\in L^\infty(\R^d)\cap L^1(\R^d)$ with $\mu_\tau\geq 0$, let $\mu$ be the solution to 
\begin{equation}\label{eq:FKequation}
\begin{cases}
-\partial_{t}\mu(t, x) - \sigma \Delta_{\GG} \mu(t, x) - \text{div}_{\GG}(\gamma|\nabla_{\GG} u(t, x)|^{\gamma - 2}\nabla_{\GG}u(t, x) \mu(t, x))=0, \quad (t, x) \in (0, \tau)\times \R^d
\\
\mu(\tau, x)=\mu_{\tau}(x), \quad x \in \R^d
\end{cases}
\end{equation}
found in \Cref{thm:ex_FP1}. Then, for any $s \in (0, \tau)$ we have that
\begin{multline}\label{eq:ABduality}
\int_{\R^{d}}{u(\tau, x) \mu(\tau, x)\ dx} = \int_{\R^{d}} u(s, x)\mu(s, x)\ dx 
\\ + \int_{s}^{\tau}\int_{\R^{d}} (\gamma - 1) |\nabla_{\GG} u(t, x)|^{\gamma} \mu(t, x)\ dtdx + \int_{s}^{\tau} \int_{\R^{d}} F(t, x)\mu(t, x) dtdx. 
\end{multline}
\end{lemma}
\proof


From the regularity of $u$ and $\mu$, by using $u$ as a test function in \eqref{eq:FKequation} and, respectively, $\mu$ as a test function in \eqref{HJweak1} and taking the sum of the two relations we obtain
\begin{multline*}
-\int_{s}^{\tau}\int_{\R^d} \partial_{t} [u(t, x) \mu(t, x)] dxdt  + \int_{s}^{\tau} \int_{\R^{d}} (\gamma |\nabla_{\GG} u(t, x)|^{\gamma} - |\nabla_{\GG} u(t, x)|^{\gamma})\mu(t, x)\ dxdt \\ = -\int_{s}^{\tau} \int_{\R^{d}} F(t, x)\mu(t, x)\ dxdt.
\end{multline*}
Hence we get \eqref{eq:ABduality}. \qed

In the following proposition we prove a key estimate using the duality argument of Lemma \ref{lem:dualityarg}.
\begin{proposition}\label{continuousbound}
Let $u \in \X(T_{0})$ be a solution to \eqref{HJJ} and let $\tau \in [0,T_{0}]$. Then, there exists $C$, depending on $T$, $\|u_{0}\|_{L^\infty(\R^{d})}$, $\|F \|_{L^\infty([0,T] \times \R^{d})}$ (and independent of $T_{0}$), such that 
\begin{equation}\label{eq:infinitybound}
\sup_{t \in [0, \tau]} \|u(t)\|_{L^\infty(\R^{d})} \leq C.
\end{equation}
\end{proposition}
\proof
First, we prove a bound from above for $u$. To do so, fix $\tau \in [0,T_{0}]$ and consider the solution $\mu: [0, \tau] \times \R^{d} \to \R$ to the following problem
\begin{equation*}
\begin{cases}
-\partial_{t} \mu(t, x) - \sigma\Delta_{\GG} \mu(t, x)= 0, & (t, x) \in [0, \tau] \times \R^{d}
\\
\mu(\tau,x)=\mu_{\tau}(x), & x \in \R^{d}
\end{cases}
\end{equation*}
with $\mu_{\tau} \in \C^{\infty}_{c}(\R^{d})\cap L^1(\R^d)$ with $\mu_\tau\geq 0$ and $\| \mu_\tau\|_{1, \R^{d}} = 1$.
By duality arguments, i.e., proceeding as in \Cref{lem:dualityarg}, we obtain 
\begin{multline}\label{eq:duality1}
\int_{\R^{d}} u(\tau, x) \mu_{\tau}(x)\ dx = \int_{\R^{d}} u_{0}(x) \mu(0, x)\ dx \\ + \int_{0}^{\tau}\int_{\R^{d}} F(s, x)\mu(s, x)\ dxds - \int_{0}^{\tau} \int_{\R^{d}} |\nabla_{\GG} u(s, x)|^{\gamma}\mu(s, x)\ dxds.
\end{multline}
Since, from \Cref{thm:ex_FP1}, $\| \mu(t, \cdot)\|_{1, \R^{d}} = 1$ for any $t \in [0, \tau]$  and $\mu\geq 0$, we get
\begin{multline*}
\int_{\R^{d}} u_{0}(x) \mu(0, x)\ dx + \int_{0}^{\tau}\int_{\R^{d}} F(s, x)\mu(s, x)\ dxds \\ - \int_{0}^{\tau} \int_{\R^{d}} |\nabla_{\GG} u(s, x)|^{\gamma}\mu(s, x)\ dxds \leq \|u_{0}\|_{L^\infty(\R^{d})} + T\|F\|_{L^\infty([0, \tau] \times \R^{d})}.
\end{multline*}
Hence, 
\begin{equation*}
\int_{\R^d} u(\tau,x) \mu_{\tau}(x)\ dx \leq \|u_{0}\|_{L^\infty(\R^{d})} + T\|F\|_{L^\infty([0, \tau] \times \R^{d})}
\end{equation*}
and, thus, by passing to the supremum, over $\mu_{\tau} \geq 0$ with $\| \mu(t, \cdot)\|_{1} = 1$ one deduces 
\begin{equation}\label{eq:upperbound}
u(\tau, x) \leq \|u_{0}\|_{L^\infty(\R^{d})} + T \|F\|_{L^\infty([0, \tau] \times \R^{d})}.
\end{equation}
To prove the lower bound for $u$, we first observe that combining \eqref{eq:ABduality} and \eqref{eq:upperbound} we get
\begin{equation}\label{comb}
\int_{0}^{\tau}\int_{\R^{d}} (\gamma - 1 ) |\nabla_{\GG} u(s, x)|^{\gamma}\mu(s, x)\ dsdx \leq 2 (\|u_{0}\|_{L^\infty(\R^{d})}  + T\|F\|_{L^\infty([0,\tau] \times \R^{d})}). 
\end{equation}
So, again by \eqref{eq:duality1} we get
\begin{multline*}
\int_{\R^{d}} u(\tau, x)\mu_{\tau}(x)\ dx \geq \int_{\R^{d}} u_{0}(x)\mu(0,x)\ dx \\ - \frac{2}{\gamma-1} (\|u_{0}\|_{L^\infty(\R^{d})}  + T\|F\|_{L^\infty([0, \tau] \times \R^{d})}) + \int_{0}^{\tau}\int_{\R^{d}}{F(s, x)\mu(s,x)\ dsdx} 
\end{multline*}
which, by the arbitrariness of $\mu_{\tau}$, yields to 
\begin{equation}\label{eq:lowerbound}
u(\tau, x) \geq -\frac{\gamma+1}{\gamma-1}\left( \|u_{0}\|_{L^\infty(\R^{d})} + T\|F\|_{L^\infty([0, \tau] \times \R^{d})}\right).
\end{equation}
Therefore, \eqref{eq:upperbound} and \eqref{eq:lowerbound} imply \eqref{eq:infinitybound}. \qed

Now we want to prove a local bound on $\nabla_{\GG} u$. To do this we need the following remark.
\begin{remarks}\label{rem:Normequivalence}\em
Following \cite{Martino_2012}, for any two Riemannian metrics $g$ and $\widetilde{g}$, we have that, by the very definition of the gradient, for every smooth function $u$ it holds that $g(D_{g} u, \cdot) = \widetilde{g}(D_{\widetilde{g}} u, \cdot)$. So, since any two Riemannian metrics are equivalent, if we denote by $|D_{g} u|_{g}^{2} = g(D_{g} u, D_{g} u)$, and the same for $D_{\widetilde{g}} u$, we have that for any compact set $\Omega \subset \R^d$ there exist constants $K_{1}(\Omega)$, $K_{2}(\Omega)$ such that 
\begin{equation}\label{eq:equivalence}
K_{1}(\Omega) |D_{g} u|_{g}^{2} \leq |D_{\widetilde{g}} u|_{\widetilde{g}}^{2} \leq K_{2}(\Omega)|D_{g} u|_{g}^{2}. 
\end{equation}
In particular, for any two sequences of metrics $\{g_{n}\}_{n \in \N}$, $\{ \widetilde{g}_{n}\}_{n \in \N}$ such that $g_{n} \to g$ and $\widetilde{g}_{n} \to \widetilde{g}$ it can be proved that $K_{1}(\Omega)$ and $K_{2}(\Omega)$ in \eqref{eq:equivalence} can be chosen independently of $n$, i.e., there exist $K_{1}^{\prime}(\Omega)$, $K_{2}^{\prime}(\Omega)$ such that 
\begin{equation}\label{eq:equivalence1}
K_{1}^{\prime}(\Omega) |D_{g_{n}} u|_{g_{n}}^{2} \leq |D_{\widetilde{g}_{n}} u|_{\widetilde{g}_{n}}^{2} \leq K_{2}^{\prime}(\Omega) |D_{g_{n}} u|_{g_{n}}^{2}. 
\end{equation}
\end{remarks}

We now recall that, to obtain a bound on $\nabla_{\GG} u$,
we cannot directly apply the classical Bernstein method to $X_i u$, where $u$ solves the equation ~\eqref{HJJ}, but we have to adapt it because some extra terms, involving commutators, would appear. In order to overcome this issue, we consider the family of right-invariant vector fields $\{Y_1, \dots, Y_m\}$ introduced in \eqref{hyp:reg_Y}. 


\begin{proposition}\label{horizontalbound}
Under the assumptions of \Cref{thm:existence1}, let $u \in \X(T_{0})$ be a solution to \eqref{HJJ} and let $\tau \in [0,T_{0}]$. Then, for any compact subset $\Omega$ of $\R^d$ there exists $C(\Omega)$, depending on $\|\nabla_{\GG}u_{0}\|_{L^\infty(\R^{d})}$, $\|\nabla_{\GG}F \|_{L^\infty([0,T_{0}] \times \R^{d})}$, such that 
\begin{equation}\label{eq:infinitybound2}
\sup_{t \in [0, \tau)} \| \nabla_{\GG} u(t)\|_{L^\infty(\Omega)} \leq C(\Omega). 
\end{equation}
\end{proposition}
\proof
 Let $\{X_{1}^{\eps}, \dots, X_{m}^{\eps}, X_{m+1}^{\eps}, \dots, X_{d}^{\eps}\}$ be a completion of $\{X_{1}, \dots, X_{m}\}$ to a Riemannian basis on $\R^{d}$ with $X_{j}^{\eps}=X_{j}$ for $j=1,\dots, m$ and, for some fields $\bar X_{j}$, $X_{j}^{\eps}=\eps \bar X_{j}$ for $j=m+1,\dots,d$. Similarly, we denote by  $\{Y_{1}^{\eps}, \dots, Y_{m}^{\eps}, Y_{m+1}^{\eps}, \dots, Y_{d}^{\eps}\}$ a completion of $\{Y_{1}, \dots, Y_{m}\}$. Observe that, for any $i$, $j = 1, \dots, d$ we have that $[X_{i}^{\eps}, Y_{j}^{\eps}] = 0$ (see~\cite[Lemma 2.1]{Martino_2012}).

 We consider the complete Hamilton-Jacobi equation 
\begin{equation}\label{eq:completesystem}
\begin{cases}
\partial_{t} u^{\eps}(t, x) - \sigma \Delta_{\eps} u^{\eps}(t, x) + |\nabla_{\eps} u^{\eps} (t, x)|^{\gamma} = F(t, x), & (t, x) \in (0,T) \times \R^{d} 
\\
u^{\eps}(0, x)=u_{0}(x), & x \in \R^{d}
\end{cases}
\end{equation}
where $\Delta_{\eps}$ and $\nabla_{\eps}$ denotes, respectively, the Laplacian and the gradient w.r.t. the vector fields $\{X_{1}^{\eps}, \dots, X_{d}^{\eps}\}$. Fix $j \in \{1, \dots, d\}$ and set $v^{\eps}_{j}=Y_{j}u^{\eps}$.   From \eqref{eq:completesystem} we have that $v^{\eps}_{j}$ is a classical solution to the equation 
\begin{equation}\label{YY}
\partial_{t} v^{\eps}_{j}(t, x) - \sigma \Delta_{\eps} v^{\eps}_{j}(t, x) + \gamma|\nabla_{\eps} u^{\eps}(t, x)|^{\gamma - 2}\nabla_{\eps} u^{\eps}(t, x) \nabla_{\eps} v^{\eps}_{j} = Y_{j}F(t, x), \quad (t, x) \in (0,T) \times \R^{d}. 
\end{equation}
We proceed with a duality argument as before. To do so, for $\tau > 0$, let $\mu$ solve
\begin{equation}\label{testY}
\begin{cases}
-\partial_{t}\mu(t, x) - \sigma \Delta_{\eps} \mu(t, x) - \text{div}_{\eps}(\gamma|\nabla_{\eps} u^{\eps}(t, x)|^{\gamma - 2}\nabla_{\eps}u^{\eps}(t, x)\mu(t, x))=0, \quad (t, x) \in (0, \tau)
\\
\mu(\tau, x)=\mu_{\tau}(x)
\end{cases}
\end{equation}
where $\text{div}_{\eps}$ denotes the divergence operator w.r.t. $\{X_{1}^{\eps}, \dots, X_{m}^{\eps}, X_{m+1}^{\eps}, \dots, X_{d}^{\eps}\}$. Thus, using $\mu$ as a test function in the weak formulation of \eqref{YY} and $v$ as a test function in the weak formulation of \eqref{testY} we obtain 
\begin{equation*}
\sup_{t \in [0, \tau)} \| v^{\eps}_{j}\|_{L^\infty(\R^{d})} \leq \| Y_{j} F\|_{L^\infty([0, \tau] \times \R^{d})} +\| Y_{j} u_0\|_{L^\infty(\R^{d})},
\end{equation*}
which is bounded by \eqref{hyp:reg_Y}.\\
Applying \Cref{rem:Normequivalence} to $\{Y_{1}^{\eps}, \dots, Y_{m}^{\eps}, Y_{m+1}^{\eps}, \dots, Y_{d}^{\eps}\}$ and $\{X_{1}^{\eps}, \dots, X_{m}^{\eps}, X_{m+1}^{\eps}, \dots, X_{d}^{\eps}\}$ we get
\begin{equation*}
\sup_{t \in [0, \tau)} \| X^{\eps}_{j} u^{\eps} \|_{L^\infty(\Omega)} \leq C(\Omega)\left(\| X_j F\|_{L^\infty([0, \tau] \times \R^{d})}  +\| X_{j} u_0\|_{L^\infty(\R^{d})}\right)
\end{equation*}
for any compact subset $\Omega$ of $\R^d$ and some constant $C(\Omega) \geq 0$. So, as $\eps \downarrow 0$ we have that $X^{\eps}_{j} u^{\eps} \to X_{j} u$ if $j \in \{1, \dots, m\}$ and $X^{\eps}_{j} u^{\eps} \to 0$ if $j \in \{m+1, \dots, d\}$. Hence, this yields to 
\begin{equation*}
\sup_{t \in [0, \tau)} \| \nabla_{\GG} u \|_{L^\infty(\Omega)} \leq C(\Omega)\left(\| \nabla_{\GG} F\|_{L^\infty([0, \tau] \times \R^{d})} +\| \nabla_{\GG} u_0\|_{L^\infty(\R^{d})}\right). \eqno\square
\end{equation*}




\vspace{0.5cm}
\noindent{\it Proof of \Cref{thm:existence1}.} Reasoning as in \cite[sect.2.1]{Davini_2019}, from \Cref{continuousbound} and \Cref{horizontalbound} we have that there exists a solution on $[0,T_0 +\eps]$ and thus, a solution $u \in C([0,T]; W^{1, \infty}_{\GG, \text{loc}}(\R^d))$ for $T$ finite but arbitrary large.

Next, in order to complete the proof of gain regularity of the solution we proceed with a bootstrap argument. For any $R \geq 0$ let $u_R$ be a solution to 
\begin{equation*}
\begin{cases}
\partial_{t} u_R(t, x) - \sigma \Delta_{\GG} u_R(t, x) + |\nabla_{\GG} u_R (t, x)|^{\gamma} = F(t, x), & (t, x) \in (0,T] \times B_R 
\\
u_R(t, x) = u_0(x), & x \in B_R 
\\
u_R(t, x)=0, & x \in \partial B_R.
\end{cases}
\end{equation*}
Then, as proved so far we have that $u_R \in C([0,T]; W^{1, \infty}_{\GG}(\R^d))$. Then, the same equation can be seen as a subelliptic heat equation with bounded right hand-side. Hence, the right hand-side belongs to $L^p(B_R)$ for any $p \geq 1$ and using \Cref{lamberton}, we obtain $u_R \in W^{1,p}([0,T]; W^{2,p}_{\GG}(B_R))$. Taking $p \geq Q+2$ applying \Cref{immersion} we gain regularity on the solution, that is, $u_R \in C^{1,\alpha}_{\GG}([0,T] \times B_R)$ with $\alpha = \frac{1}{p}(p - (Q+2))$. So, following again the same reasoning we deduce that for any $R \geq 0$ the solution $u_R$ belongs to $C^{2+\alpha, 1+\frac{\alpha}{2}}_{\GG}([0,T] \times B_R)$. Finally, as $R \uparrow \infty$ by a diagonalization argument the proof is complete. \qed


\section{Application to MFG}

The goal of this section is the application of the above results to get the existence for small times of solutions to the following MFG system
\begin{equation}\label{SR-MFG}
\begin{cases}
-\partial_{t} u(t, x) - \sigma \Delta_{\GG} u(t, x) + |\nabla_{\GG} u (t, x)|^{\gamma} = F_{\textrm{\tiny {MFG}}}[\rho](x), & (t, x) \in (0,T) \times \R^{d} 
\\
\partial_{t} \rho_{t} - \sigma \Delta_{\GG} \rho - \text{div}_{\GG}(\gamma |\nabla_{\GG} u(t, x)|^{\gamma-2} \nabla_{\GG}u(t, x)\rho)=0, & (t, x) \in (0,T) \times \R^{d}
\\
u(T, x) = u_{T}(x), \quad \rho(0, x)=\rho_0(x), & x \in \R^d. 
\end{cases}
\end{equation}
Note that we cannot prove the existence for any time $T$ because of the lack of compactness in the results of \Cref{thm:existence1}.\\
As it is customary in MFG, the existence result is a consequence of the Schauder's fixed point theorem. However, due to the lack of control of the moments of the measure $\rho_t$ associated to the solution of the Fokker-Plank equation, the strategy is quite different w.r.t. the classical literature.


\medskip
\noindent{\it Proof of \Cref{MFGmain}.} We obtain the existence of a classical solution using the Schauder fixed-point theorem. To do so, we endow the space $C([0,T]; W^{2,\infty}_{\GG}(\R^{d}))$  with the topology induced by the uniform convergence and we construct a map 
\begin{equation*}
\TT: C([0,T]; W^{2,\infty}_{\GG}(\R^{d}))\to C([0,T]; W^{2,\infty}_{\GG}(\R^{d}))
\end{equation*}
in the following way: given $u \in \mathcal{C}$, with
\begin{equation*}
\mathcal{C}= \left\{ u \in C([0,T]; W^{2,\infty}_{\GG}(\R^{d})) : \sup_{t \in [0,T]} \|u(t)\|_{W^{2,\infty}_{\GG}(\R^d)} \leq \kappa(T_0)\right\}
\end{equation*} 
where $T_0$ and $\kappa(T_0)$ are the constants constructed in \Cref{thm:existence} and let $\mu$ be the unique solution constructed in \Cref{thm:ex_FP1} to 
\begin{equation*}
\begin{cases}
\partial_{t} \mu - \sigma \Delta_{\GG} \mu - \text{div}_{\GG}(\gamma |\nabla_{\GG} u(t, x)|^{\gamma-2} \nabla_{\GG}u(t, x)\mu)=0, & (t, x) \in (0,T) \times \R^{d}
\\
\mu(0,x)=\rho_0(x), & x \in \R^d.
\end{cases}
\end{equation*}
Then, we set $\psi = \TT(u)$ as the unique solution constructed in \Cref{thm:existence} to the Hamilton-Jacobi equation
\begin{equation*}
\begin{cases}
-\partial_{t} \psi(t, x) - \sigma \Delta_{\GG} \psi(t, x) + |\nabla_{\GG} \psi (t, x)|^{\gamma} = F_{\textrm{\tiny {MFG}}}[\mu](x), &  (t, x) \in (0,T) \times \R^{d}
\\
\psi(T, x) = u_T(x), & x \in \R^d. 
\end{cases}
\end{equation*}

First, we claim that the map $\TT$ is well-defined and continuous for any time horizon $T \leq T_0$. Indeed, the vector field $|\nabla_{\GG} u|^{\gamma-2} \nabla_{\GG} u$ satisfies the assumptions of \Cref{thm:ex_FP1} and the solution to the Fokker-Plank equation is H\"older continuous by  \Cref{holder}. Moreover, from \Cref{thm:existence} we have that $\TT(\mathcal{C})$ is a compact subset of $\mathcal{C}$ since the solution to the Hamilton-Jacobi equation is locally bounded in $C^{2+\alpha, 1+\frac{\alpha}{2}}_{\GG}([0,T] \times \R^{d}))$ by \cite[Theorem 1.1]{BramantiBrandolini_07}. 

Hence, since all the assumptions of the Schauder fixed point theorem are satisfied, the proof is complete.  
\qed


%
%
%
%


\begin{thebibliography}{10}

\bibitem{Achdou_2022-1}
Y. Achdou, P. Mannucci, C. Marchi, N.Tchou.
\newblock Deterministic mean field games with control on the acceleration.
\newblock {NoDEA Nonlinear Differential Equations Appl.}, 27 (2020), no. 3, p. 33.

\bibitem{Achdou_2022}
Y. Achdou, P. Mannucci, C. Marchi, N.Tchou.
\newblock Deterministic mean field games with control on the acceleration and
  state constraints.
\newblock {\em {SIAM} J. Math.Anal.}, 54  (2022), (3), 3757--3788.

\bibitem{Alexopoulos_2002}
G.~K. Alexopoulos.
\newblock Sub-laplacians with drift on Lie groups of polynomial volume growth.
\newblock {\em Mem. Amer. Math. Soc.}, 155 (739), 2002.

  
  \bibitem{ABA}
L. Amour, M. Ben-Artzi
\newblock Global existence and decay for viscous Hamilton-Jacobi equations. 
\newblock {Nonlinear Anal.}, 31, (1998), no. 5-6, 621--628.
  
  
\bibitem{Baspinar_2020}
E.~Baspinar, A.~Sarti, G.~Citti.
\newblock A sub-Riemannian model of the visual cortex with frequency and phase.
\newblock {\em J. Math. Neurosci.}, 10 (2020), 11, 31 pp. 10 

\bibitem{Bell_2004}
D.~R. Bell.
\newblock Stochastic differential equations and hypoelliptic operators.
\newblock In {\em Real and stochastic analysis. New perspectives.}, 
  9--42, Boston, MA: Birkh{\"a}user, 2004.

\bibitem{bib:BFY}
A.~Bensoussan, J.~Frehse, P.~Yam.
\newblock {\em Mean field games and mean field type control theory}.
\newblock Springer Briefs in Mathematics. Springer, New York 2013.


\bibitem{Biagi_2018}
S. Biagi, A. Bonfiglioli.
\newblock {\em  An introduction to the geometrical analysis of vector fields--with applications to maximum principles and Lie groups}.
\newblock {World Scientific Publishing}, Hackensack, NJ, 2019




\bibitem{Bonfiglioli}
A.~Bonfiglioli, E.~Lanconelli, F.~Uguzzoni.
\newblock {\em Stratified Lie Groups and Potential Theory for their
  Sub-Laplacians}. \newblock Springer Monographs in Mathematics, Springer, Berlin 2007.

\bibitem{Bonfiglioli_2003}
A. Bonfiglioli, E. Lanconelli, F. Uguzzoni.
\newblock Fundamental solutions for non-divergence form operators on stratified
  groups.
\newblock {\em Trans. Amer. Math. Soc.},
  356 (2004) 7, 2709--2737.
  
  
\bibitem{BramantiBrandolini_07}
M.~Bramanti, L.~Brandolini.
\newblock Schauder estimates for parabolic nondivergence operators of
H\"ormander type.
\newblock {\em J. Differential Equations}, 234, (2007), 177--245.

\bibitem{Bramanti_2000}
M. Bramanti, L. Brandolini.
\newblock {{\(L^p\)}} estimates for nonvariational hypoelliptic operators with
  {{\(VMO\)}} coefficients.
\newblock {\em Trans. Amer. Math. Soc.}, 352 (2000), 2, 781--822.

\bibitem{Bramanti_2010}
M. Bramanti, L. Brandolini, E. Lanconelli, F. Uguzzoni.
\newblock Non-divergence equations structured on H{\"o}rmander vector fields:
  heat kernels and Harnack inequalities.
\newblock {\em Mem. Amer. Math. Soc.}, 204, (2010) n. 961.



\bibitem{Brezis}
H.~Brezis.
\newblock {\em Analyse fonctionelle}, volume Collection Math\'ematique pour la
  ma\^itrise.
\newblock Masson, 1987.

\bibitem{bib:CM}
P.~Cannarsa, C.~Mendico.
\newblock Mild and weak solutions of mean field game problems for linear
  control systems.
\newblock {\em Minimax Theory Appl., 5 (2020) , 2, 221--250},
  2020.
  



\bibitem{bib:NC}
P.~Cardaliaguet.
\newblock Notes on Mean Field Games, from P.L. Lions lectures at College de France (2012), available at https://www.ceremade.dauphine.fr/cardalia/MFG20130420.pdf.

%

\bibitem{Cinti_2009}
C. Cinti.
\newblock Partial Differential Equations--Uniqueness in the Cauchy problem for a class of hypoelliptic
  ultraparabolic operators.
\newblock {\em Atti Accad. Naz. Lincei Rend. Lincei Mat. Appl.}, 120 (2009), no. 2, 145--158.  
  

\bibitem{Cirant_2020}
M. Cirant, A. Goffi.
\newblock Lipschitz regularity for viscous Hamilton-Jacobi equations with Lp
  terms.
\newblock {\em Ann. Inst. H. Poincar{\'{e}} C,
  Anal. Non Lin{\'{e}}aire}, 37,  (2020) 4, 757--784.
    
  




\bibitem{2014}
G. Citti, A. Sarti, editors.
\newblock {\em Neuromathematics of Vision}.
\newblock Springer Berlin Heidelberg, 2014.

\bibitem{Danielli_1996}
D. Danielli, L. Capogna, N. Garofalo.
\newblock Capacitary estimates and the local behavior of solutions of nonlinear
  subelliptic equations.
\newblock {\em Amer. J. of Math.}, 118, (1996), 6, 1153--1196, .

\bibitem{Davini_2019}
A. Davini.
\newblock Existence and uniqueness of solutions to parabolic equations with
  superlinear hamiltonians.
\newblock {\em Commun. Contemp. Math.}, 21,  (2019) 1, 1750098.

\bibitem{Dragoni_2018}
F. Dragoni, E.Feleqi.
\newblock Ergodic mean field games with H{\"o}rmander diffusions.
\newblock {\em Calc. Var. Partial Differential Equations},
 57  (2018) 5, , 22 pp.

\bibitem{Ersland_2021}
O. Ersland, E.~R. Jakobsen.
\newblock On fractional and nonlocal parabolic mean field games in the whole
  space.
\newblock {\em J. Differential Equations}, 301, (2021), 428--470.



\bibitem{Evans_2003}
L.C. Evans.
\newblock Some new {PDE} methods for weak {KAM} theory.
\newblock {\em Calc. Var. Partial Differential Equations},
  17, (2003) 2, 159--177.

\bibitem{Gomes_2020}
E. Feleqi, D. Gomes, T. Tada.
\newblock Hypoelliptic mean field games -- a case study.
\newblock {\em Minimax Theory Appl.}, 5, (2020) 2, 305--326.

\bibitem{Franchi_1995}
B.~Franchi, G.~Lu, R.~L. Wheeden.
\newblock Weighted Poincar\'e inequalities for H\"ormander vector fields and
  local regularity for a class of degenerate elliptic equations.
\newblock {\em Potential Analysis}, 4, (1995) 4, 361--375.

\bibitem{Frentz_2012}
M. Frentz, E. G{\"o}tmark, K.Nystr{\"o}m.
\newblock The obstacle problem for parabolic non-divergence form operators of
 H{\"o}rmander type.
\newblock {\em J. Differential Equations}, 252, (2012) 9, 5002--5041.


\bibitem{bib:DEV}
D.~A. Gomes, E.~A. Pimentel, V.~Voskanyan.
\newblock {\em Regularity theory for mean-field game systems}.
\newblock SpringerBriefs in Mathematics. Springer, Berlin 2016.


\bibitem{Gomes_2016}
D.~A. Gomes, E.Pimentel, H. S{\'{a}}nchez-Morgado.
\newblock Time-dependent mean-field games in the superquadratic case.
\newblock {\em {ESAIM} Control Optim. Calc. Var.},
  22, (2016) 2, 562--580.


%


\bibitem{Lamberton_1987}
D. Lamberton.
\newblock Equations d'{\'{e}}volution lin{\'{e}}aires
  associ{\'{e}}es {\`{a}} des semi-groupes de contractions dans les espaces lp.
\newblock {\em J. Funct. Anal.}, 72 (1987) 2, 252--262.

\bibitem{bib:LL1}
J.-M. Lasry, P.-L. Lions.
\newblock Jeux \'a champ moyen. i. le cas stationnaire.
\newblock {\em C.R. Math. Acad. Sci. Paris}, 619--625, 2006.

\bibitem{bib:LL2}
J.-M. Lasry, P.-L. Lions.
\newblock Jeux \`a champ moyen. ii. horizon fini et controle optimal. 
\newblock {\em C.R. Math. Acad. Sci. Paris}, (2006), 679--684.

\bibitem{bib:LL3}
J.-M. Lasry, P.-L. Lions.
\newblock Mean field games.
\newblock {\em Jpn. J. Math.}, 2, (1), (2007) 229--260.


\bibitem{Lu_1992}
G. Lu.
\newblock Weighted Poincar{\'{e}} and Sobolev inequalities for vector fields
  satisfying H{\"o}rmander's condition and applications.
\newblock {\em Rev. Mat. Iberoamericana}, (1992) 367--439.

\bibitem{Mannucci_2020}
P. Mannucci, C. Marchi, C. Mariconda, N. Tchou.
\newblock Non-coercive first order mean field games.
\newblock {\em J. Differential Equations}, 269(5), (2020) 4503--4543.

\bibitem{Mannucci_2022}
P. Mannucci, C. Marchi, N. Tchou.
\newblock Non coercive unbounded first order mean field games: The Heisenberg
  example.
\newblock {\em J. Differential Equations}, 309, (2022), 809--840.

\bibitem{Martino_2012}
V. Martino,  A. Montanari.
\newblock Lipschitz continuous viscosity solutions for a class of fully
  nonlinear equations on lie groups.
\newblock {\em J. Geom. Anal.}, 24(1), (2014), 169--189.

\bibitem{Mendico}
C. Mendico.
\newblock A singular perturbation problem for mean field games of acceleration: application to mean field games of control
\newblock {\em to appear in J. Evol. Equ.} (2023)


\bibitem{Montgomery}
R.~{Montgomery}.
\newblock {\em {A tour of SubRiemannian Geometries, their Geodesics and
  Applications}}, volume 91.
\newblock AMS Providence,  2002.




\bibitem{Paoli_2016}
E. Paoli.
\newblock Small time asymptotics on the diagonal for H{\"o}rmander's type
  hypoelliptic operators.
\newblock {\em J. Dyn. Control Syst.}, 23(1), (2016),111--143.
  

\bibitem{Polidoro_1994}
S. Polidoro.
\newblock On a class of ultraparabolic operators of
  {Kolmogorov}-{Fokker}-{Planck} type.
\newblock {\em Matematiche}, 49 (1), (1994), 53--105.

\bibitem{RS_77}
L.P. Rothschild, E.M. Stein.
\newblock Hypoelliptic differential operators and nilpotent groups.
\newblock {\em Acta Math.}, 137 (1976), no. 3-4, 247--320.


\bibitem{bib:CV}
C.~Villani.
\newblock {\em Topics in optimal transportation}.
\newblock Graduate Studies in Mathematics, {\bf 58}. American Mathematical
  Society, Providence, RI, 2003.


\end{thebibliography}

\end{document}